\documentclass{article}[12pt]
\usepackage{amsmath}
\usepackage{amsfonts}
\usepackage{amssymb}
\usepackage{appendix}
\usepackage{amsthm }
\usepackage{footnote}
\usepackage{fullpage}

\usepackage[numbers,colon]{natbib}

\usepackage{enumerate}
\newtheorem{Thm}{Theorem}[section]

\newtheorem{Lem}[Thm]{Lemma}
\newtheorem{Cor}[Thm]{Corollary}
\newtheorem{Pro}[Thm]{Proposition}

\makeatletter
\def\blfootnote{\xdef\@thefnmark{}\@footnotetext}
\makeatother

\theoremstyle{definition}
\newtheorem{Def}[Thm]{Definition}
\newtheorem{Eg}[Thm]{Example}

\theoremstyle{remark}
\newtheorem{Rem}[Thm]{\bf{Remark}}

\newcommand{\ConvD}{\overset{d}{\rightarrow}}
\newcommand{\ConvFDD}{\overset{f.d.d.}{\longrightarrow}}

\newcommand{\Cov}{\mathrm{Cov}}
\newcommand{\Var}{\mathrm{Var}}
\newcommand{\E}{\mathbb{E}}
\newcommand{\mathbd}{\boldsymbol}

\title{Convergence of long-memory discrete $k$-th order Volterra processes}
\author{
        Shuyang Bai ~~
        Murad S. Taqqu
}

\begin{document}
\maketitle
\begin{abstract}
We obtain  limit theorems for a class of nonlinear discrete-time processes $X(n)$ called the $k$-th order Volterra processes of order $k$. These are moving average $k$-th order polynomial forms:
\[
X(n)=\sum_{0<i_1,\ldots,i_k<\infty}a(i_1,\ldots,i_k)\epsilon_{n-i_1}\ldots\epsilon_{n-i_k},
\]
where $\{\epsilon_i\}$ is i.i.d.\ with $\E \epsilon_i=0$, $\E \epsilon_i^2=1$,  where $a(\cdot)$ is a nonrandom coefficient, and where the diagonals are included in the summation.  We specify conditions for $X(n)$ to be well-defined in $L^2(\Omega)$, and focus on central and non-central limit theorems. We show that normalized partial sums of centered $X(n)$ obey the central limit theorem if $a(\cdot)$ decays fast enough so that $X(n)$ has short memory. We prove a non-central limit theorem if, on the other hand, $a(\cdot)$ is asymptotically some slowly decaying homogeneous function  so that $X(n)$ has long memory. In the non-central case the  limit is a linear combination of Hermite-type processes of different orders. This linear combination can  be expressed as a centered multiple Wiener-Stratonovich integral.
\end{abstract}
\blfootnote{
\begin{flushleft}
\textbf{Key words} Long memory; Long-range dependence; Volterra process; Wiener chaos; Wiener; Stratonovich; Limit theorems
\end{flushleft}
\textbf{2010 AMS Classification:} 60G18, 60F05\\
}
\section{Introduction}
A common assumption when analyzing a  stationary time series $\{X(n),n\in \mathbb{Z}\}$, is that $\{X(n)\}$ is a causal linear process, that is,
\begin{equation}\label{eq:linear}
X_1(n)=\sum_{i=1}^\infty a_{i}\epsilon_{n-i},
\end{equation}
where $\{\epsilon_i\}$ is a sequence of i.i.d.\ random variables with mean $0$ and variance $1$. This assumption is based on the Wold's decomposition, which states that if $\{X(n)\}$ is stationary with mean $0$ and finite second moment, and is also purely non-deterministic, then the representation (\ref{eq:linear}) always holds with $\{\epsilon_{i}\}$  a sequence of uncorrelated random variables (\citet{brockwell:1991:time} \S 5.7). The  independence assumption of $\{\epsilon_i\}$  in (\ref{eq:linear}) obliterates the higher-order dependence structure. In some applications, linear processes provide  good approximations, while in  others, not, as in the case of the ARCH model for  volatility data.

The Volterra process extends linear process by incorporating non-linearity. A (causal) Volterra process with highest order $K$ is of the form
\begin{equation}\label{eq:volterra}
X_K(n)=\sum_{k=1}^K\sum_{0<i_1,\ldots,i_k<\infty} a_k(i_1,\ldots,i_k) \epsilon_{n-i_1}\ldots \epsilon_{n-i_k}.
\end{equation}

To understand the importance of (\ref{eq:volterra}), suppose that the stationary process is  $X(n)=A(\epsilon_{n-1},\epsilon_{n-2},\ldots)$ for some regular function $A$. Then (\ref{eq:volterra}) can be heuristically regarded as its $K$-th order Taylor series approximation. The homogeneous polynomial-form expansion in (\ref{eq:volterra}) and its continuous-time counterpart where the sums are replaced with integrals,  was originally proposed by Vito Volterra (see \citet{volterra:2005:theory}) for modeling deterministic nonlinear systems, and later extended by Norbert Wiener (see \citet{wiener:1966:nonlinear}) to random systems, which eventually lead to the well-developed theory of Wiener chaos (see, e.g., \citet{cameron:1947:orthogonal}, \citet{ito:1951:multiple}, and the recent survey \citet{peccati:taqqu:2011:wiener}). In the context of approximation of stationary processes, \citet{nisio:1960:polynomial} shows that any stationary process can be approximated in the sense of finite-dimensional distributions by a Volterra process with $\epsilon_i$'s Gaussian. Some nonlinear time series models admit  Volterra expansions (\ref{eq:volterra}) with $K=\infty$. For example,  the LARCH($\infty$) model
\[
X(n)=a+\sum_{i=1}^\infty b_i Y(n-i), \quad Y(n)=X(n)\epsilon_n,
\]
under suitable conditions admits the following Volterra expansion (see, e.g., Theorem 2.1 of \citet{giraitis:leipus:2004:larch}):
\[
X(n)=a\left(1+\sum_{k=1}^\infty \sum_{0<i_1,\ldots,i_k<\infty}b_{i_1}b_{i_2-i_1}b_{i_3-i_2}\ldots b_{i_{k}-i_{k-1}}1_{\{i_1<i_2<\ldots<i_{k}\}} \epsilon_{n-i_1}{\epsilon_{n-i_2}}\ldots \epsilon_{n-i_k}\right).
\]

We are interested  here in stationary processes that have \emph{long memory}, or \emph{long-range dependence}. A common choice is a linear process in (\ref{eq:linear}) with $a_1(n)\sim c n^{d-1}$ as $n\rightarrow\infty$, where $d\in (0,1/2)$ is the \emph{memory parameter}, and $c>0$ is some constant.  This is the case, for instance, when $X(n)$ is the stationary solution of the fractional difference equation
\[
\Delta^d X(n) =\epsilon_{n-1},
\]
where $\Delta=I-B$ is the difference operator with $I$ being identity operator and $B$ being the backward shift operator, and $(I-B)^d$ is understood as a binomial series (see, e.g., \citet{giraitis:koul:surgailis:2009:large} Chapter 7.2). We note that such long-memory linear processes have an autocovariance decaying like $n^{2d-1}$ as $n\rightarrow\infty$, and a spectral density exploding  at the origin as $|\lambda|^{-2d}$ as $|\lambda|\rightarrow 0$.

If one wants to consider a nonlinear long memory model, a natural choice is to have a Volterra process (\ref{eq:volterra}) with  coefficients $a_k(i_1,\ldots,i_k)$ decaying slowly as $i_1,\ldots,i_k$ tends to infinity, so that the autocovariance has a slow hyperbolic decay.
The major goal in this paper is to study the limit  of normalized partial sum of some long-memory Volterra processes. When $X(n)$ is a long-memory linear process, that is, a long-memory Volterra process with $K=1$, then the limit, as is well-known, is fractional Brownian motion (\citet{davydov:1970:invariance}). When $X(n)$ is polynomial of a long memory linear processes, that is, when $a_k(i_1,\ldots,i_k)=c_k i_1^{d-1}\ldots i_k^{d-1}$ in (\ref{eq:volterra}) for some constant $c_k$, and $d$ is large enough, then the limit is a Hermite process of a fixed order (\citet{surgailis:1982:zones}, \citet{avram:1987:noncentral}). Such limit theorems involving non-Brownian motion limits are often called \emph{non-central limit theorems}.

In this paper, we  focus on  Volterra processes  of a single order $k\ge 1$:
\begin{equation}\label{eq:X(n)}
X(n)=\sum_{0<i_1,\ldots,i_k<\infty} a(i_1,\ldots,i_k) \epsilon_{n-i_1}\ldots \epsilon_{n-i_k},
\end{equation}
which avoids possible cancellations between terms of different orders.
Note that the multiple sum (\ref{eq:X(n)}) \emph{includes} diagonals, that is, it allows $i_1,\ldots,i_k$ to be equal to each other.  In the literature, one often considers multiple sums of the type (\ref{eq:X(n)}) where summation over the diagonals is \emph{excluded}, which greatly simplifies the theory. Although the exclusion of the diagonals is a typical theoretical assumption, it is, from a practical perspective, an artificial one. Expression (\ref{eq:X(n)}) is the natural one since it includes all the terms.

To obtain a non-central limit theorem for (\ref{eq:X(n)}),
we assume that the coefficient $a(i_1,\ldots,i_k)$ behaves asymptotically as a homogeneous function $g$ on $\mathbb{R}_+^k$ which is bounded excluding a neighborhood of the origin.  We shall show that in this case, the limit of a normalized sum of centered $X(n)$ is a linear combination of Hermite-type processes of different orders.
These Hermite-type processes that appear in the limit were first introduced in \citet{mori:toshio:1986:law}, and were called in \citet{bai:taqqu:2013:generalized}   \emph{generalized Hermite processes}. They live in  Wiener chaos, and extend in a natural way the usual Hermite processes considered in the literature, e.g., \citet{dobrushin:major:1979:non} and \citet{taqqu:1979:convergence}.

The limit, which is a linear combination involving different orders of multiple Wiener-It\^o integrals, can be re-expressed as a single  \emph{centered multiple Wiener-Stratonovich  integral} with the zeroth-order term excluded. These integrals were introduced by \citet{hu:meyer:1988:integrales}. Loosely speaking, in contrast to the usual Wiener-It\^o integrals, the multiple Wiener-Stratonovich integrals include  diagonals, and intuitively they are the continuous counterpart  of the multiple sums in (\ref{eq:X(n)}) which, as was noted, do include diagonals.

The paper is organized as follows. In Section \ref{Sec:GHK}, we introduce the generalized Hermite processes which  appear in the formulation of the non-central limit theorem. In Section \ref{Sec:inf homo sum}, we provide conditions for the polynomial form  (\ref{eq:X(n)}) to be well-defined in $L^2(\Omega)$. In Section \ref{Sec:LRD Volterra}, we introduce the class of long-memory Volterra processes $X(n)$ of interest in the non-central limit theorem. In Section \ref{Sec:CLT}, we establish central limit theorems when $a(\cdot)$ in (\ref{eq:X(n)}) decays fast enough so that $X(n)$ has short memory. In Section \ref{Sec:NCLT}, we state a  non-central limit theorem for processes $X(n)$ in (\ref{eq:X(n)}). Before launching into the article, the reader may want to have a look at  this result, formulated as  Theorem \ref{Thm:Basic NCLT}, and also at the  illustrative Example \ref{Eg}. The connection between the limit and multiple Wiener-Stratonovich integrals is indicated    in Section \ref{Sec:Stra}.  Section \ref{Sec:Extension hyper} contains an extended hypercontractivity formula.

\section{Generalized Hermite processes and  kernels}\label{Sec:GHK}
We introduce here the kernels which will be used to define both the coefficient $a(\cdot)$ in (\ref{eq:X(n)}), and the processes that will appear in the non-central limit.

First, some notation which will be used throughout the paper.  Let $\mathbf{x}=(x_1,\ldots,x_k)\in \mathbb{R}^k$,  $\mathbf{i}=(i_1,\ldots,i_k)\in \mathbb{Z}^k$, $\mathbf{0}=(0,\ldots,0)$,  $\mathbf{1}=(1,\ldots,1)$, and let $\mathbf{1}_r$ denote the vector made of $r$ $1$'s. If $x\in \mathbb{R}$, then $[x]=\sup \{n\in \mathbb{Z},n\le x\}$, and $[\mathbf{x}]=([x_1],\ldots,[x_k])$. We write $\mathbf{x}> \mathbf{y}$ if $x_j> y_j, j=1,\ldots,k$, and use the following standard notations: $\|\cdot\|$ denotes a norm in some suitable space, $\mathrm{1}_A(\cdot)$ is the indicator function of a set $A$, $|A|$ denotes the cardinality of set $A$, and if $g_1$ and  $g_2$ are two functions on $\mathbb{R}^{k_1}$ and $\mathbb{R}^{k_2}$ respectively, then $g_1\otimes g_2$ defines a scalar function on $\mathbb{R}^{k_1+k_2}$ as $(\mathbf{x}_1,\mathbf{x}_2)\rightarrow g_1(\mathbf{x}_1)g_2(\mathbf{x}_2)$.

The following class of functions was introduced in \citet{bai:taqqu:2013:generalized}:
\begin{Def}\label{Def:GHK}
A \emph{generalized Hermite kernel} (GHK) $g$ is a nonzero measurable function defined on $\mathbb{R}_+^k$  satisfying:
\begin{enumerate}
\item $g(\lambda \mathbf{x})=\lambda^\alpha g(\mathbf{x})$, $\forall \lambda>0$,  $\alpha\in(-\frac{k+1}{2},-\frac{k}{2})$;\label{ass:homo}
\item $\int_{\mathbb{R}_+^k}|g(\mathbf{x})g(\mathbf{1}+\mathbf{x})| d\mathbf{x}  <\infty$.\label{ass:int2}
\end{enumerate}
\end{Def}
\begin{Rem}\label{Rem:int|g(s1-x)|ds<inf}
As  shown in Theorem 3.5 and Remark 3.6 in \citet{bai:taqqu:2013:generalized}, if $g(\cdot)$ is a GHK on $\mathbb{R}^k$, then for every $t>0$,  
$$\int_{0}^t |g(s\mathbf{1}-\mathbf{y})|\mathrm{1}_{\{s\mathbf{1}>\mathbf{y}\}}ds <\infty
$$
 for a.e.\ $\mathbf{y}\in \mathbb{R}^k$. Furthermore,
$$h_t(\mathbf{y})=\int_{0}^t g(s\mathbf{1}-\mathbf{y})\mathrm{1}_{\{s\mathbf{1}>\mathbf{y}\}}ds$$
 is a.e.\ defined, and $h_t\in L^2(\mathbb{R}^k)$. In addition, if $g$ is nonzero, then $\int_{\mathbb{R}_+^k}g(\mathbf{1}+\mathbf{x})g(\mathbf{x})d\mathbf{x}>0$.
\end{Rem}

These functions $g$ were used in \citet{bai:taqqu:2013:generalized} as defining kernels for a class of stochastic processes  called \emph{generalized Hermite processes}.
\begin{Def}\label{Def:GHK proc}
The \emph{generalized Hermite processes} are defined through the following multiple Wiener-It\^o integrals:
\begin{equation}\label{eq:gen herm proc}
Z(t)=I_k(h_t):=\int_{\mathbb{R}^k}' \int_0^t~g(s\mathbf{1}-\mathbf{x})1_{\{s\mathbf{1}>\mathbf{x}\}}ds ~ B(dx_1)\ldots B(dx_k) ，
\end{equation}
where the prime $'$ indicates that one does not integrate on the diagonals $x_p=x_q$, $p\neq q \in \{1,\ldots,k\}$, $B(\cdot)$ is a Brownian random measure, and $g$ is a GHK defined in Definition \ref{Def:GHK}.
\end{Def}

The generalized Hermite processes are self-similar with Hurst exponent
\begin{equation}\label{eq:H}
H=\alpha+k/2+1\in (1/2,1),
\end{equation}
 that is, $\{Z(\lambda t), t>0 \}$ has the same finite-dimensional distributions as $\{\lambda^H Z(t), t>0\}$, and they have also stationary increments.
\begin{Eg}
When $g$ takes the particular form $g(\mathbf{x})=\prod_{j=1}^k x_j^{\alpha/k}$ where $\alpha/k\in \left(-\frac{1}{2}(1+\frac{1}{k}),-\frac{1}{2}\right)$, $Z(t)$ becomes the usual Hermite process obtained  through a non-central limit theorem in the context of long memory (e.g., \citet{taqqu:1979:convergence}, \citet{dobrushin:major:1979:non}, \citet{surgailis:1982:zones}).
\end{Eg}
In  \citet{bai:taqqu:2013:generalized} the following  subclass of functions $g$, called \emph{generalized Hermite kernel of Class (B)} was considered.
\begin{Def}\label{Def:class bounded}
We say that a nonzero homogeneous function $g$  on $\mathbb{R}_+^k$ having homogeneity exponent $\alpha$ is of Class (B) (abbreviated as ``GHK(B)'', ``B'' stands for ``boundedness''), if
\begin{enumerate}
\item $g$ is a.e.\ continuous on $\mathbb{R}_+^k$;
\item $|g(\mathbf{x})|\le C\|\mathbf{x}\|^{\alpha}$ for some constant $C>0$, where $\alpha$ is as in Definition \ref{Def:GHK}.
\end{enumerate}
\end{Def}

\begin{Rem}\label{Rem:bound prod}
The norm $\|\cdot\|$ in Definition \ref{Def:class bounded} can be any norm in the finite-dimensional space $\mathbb{R}^k$ since all the norms are equivalent. For convenience, we  choose throughout this paper $\|\mathbf{x}\|=\sum_{j=1}^k |x_j|$. The GHK(B) class is a subset of the GHK class, because if $g$ is a GHK(B),  then it is homogeneous and hence satisfies Condition \ref{ass:homo} of Definition \ref{Def:GHK}.  It also satisfies Condition \ref{ass:int2} of Definition \ref{Def:GHK}. Indeed,  we have for some $C, C'>0$ that
\begin{equation}\label{eq:g bound prod}
|g(\mathbf{x})|\le C \|\mathbf{x}\|^{\alpha}=C\left(\sum_{j=1}^k x_j\right)^\alpha \le C' \prod_{j=1}^k x_j^{\alpha/k},~ \mathbf{x}\in \mathbb{R}_+^k,
\end{equation}
where the last inequality follows from the arithmetic-geometric mean inequality 
$$
k^{-1}\sum_{j=1}^ky_j\ge \left(\prod_{j=1}^k y_j\right)^{1/k} \quad \mbox{\rm  and}
\quad \alpha<0.
$$
 In view of Condition \ref{ass:homo} of  Definition \ref{Def:GHK}, since $-1\le -1/2-1/(2k)<\alpha/k<-1/2$,  we hence have
\begin{align*}
\int_{\mathbb{R}_+^k} |g(\mathbf{x})g(\mathbf{1}+\mathbf{x})|d\mathbf{x}\le C'\left(\int_0^\infty x^{\alpha/k}(1+x)^{\alpha/k}dx\right)^k<\infty.
\end{align*}
\end{Rem}
\begin{Eg}
As an example of a GHK(B), we can simply set $g(\mathbf{x})$ equal to
\[
g_1(\mathbf{x})=\|\mathbf{x}\|^{\alpha}=|x_1+\ldots+x_k|^{\alpha}=(x_1+\ldots+x_k)^{\alpha},\quad\alpha\in (-\frac{k+1}{2},-\frac{k}{2}),
\]
since $\mathbf{x}\in \mathbb{R}_+^k$.
\end{Eg}

\begin{Eg}
As another example, consider 
$$
g_2(\mathbf{x})=\prod_{j=1}^k x_j^{a_j}/(\sum_{j=1}^k x_j^b),\quad a_j>0,\ b>0,
 $$
 and 
 $$
 \sum_{j=1}^k a_j-b\in (-\frac{k+1}{2},-\frac{k}{2}).
  $$
  $g_2$ is continuous and homogeneous with exponent 
  $\alpha= \sum_{j=1}^k a_j-b$.
   It is a GHK(B) because the functions $\mathbf{x}\rightarrow \prod_{j=1}^k x_j^{a_j}$ and $\mathbf{x}\rightarrow (\sum_{j=1}^k x_j^b)^{-1}$ are bounded on the $k$-dimensional unit sphere restricted to $\mathbb{R}_+^k$. For instance, 
   $$
   \Big(\sum_{j=1}^k x_j^b\Big)^{1/b} \le C \|\mathbf{x}\|
    $$
    by the equivalence of norms on $\mathbb{R}^k$. Thus $g_2(\mathbf{x})\le C\|\mathbf{x}\|^{\alpha}$.
\end{Eg}

\begin{Eg}
It is easy to see  that  the set of GHK(B) functions on $\mathbb{R}_+^k$ with fixed homogeneity exponent $\alpha$ (with the zero function added) is closed under  linear combinations  and taking maximum  or minimum.  Thus one can consider   $g_1+g_2$, $g_1\vee g_2$ and $g_1\wedge g_2$ using the $g_1$ and $g_2$ in the foregoing examples.
\end{Eg}

In \citet{bai:taqqu:2013:generalized}, non-central limit theorems involving   GHK(B) are established\footnote{In \citet{bai:taqqu:2013:generalized}, the non-central limit theorem is  shown to hold for a larger class of functions which includes functions like $g(\mathbf{x})=\prod_{j=1}^k x_j^{\alpha/k}$, called Class (L). We do not  consider this class here, since the main result Theorem \ref{Thm:Basic NCLT} below does not hold for Class (L) in general.}.
These theorems involve sums of a long-memory stationary process  called \emph{discrete chaos process} defined as
\begin{equation}\label{eq:off diagonal}
X'(n)=\sum_{\mathbf{i}\in \mathbb{Z}_+^k}' a(i_1,\ldots,i_k) \epsilon_{n-i_1}\ldots \epsilon_{n-i_k}=\sum_{\mathbf{i}\in \mathbb{Z}_+^k}' a(\mathbf{i}) \epsilon_{n-i_1}\ldots \epsilon_{n-i_k},
\end{equation}
where $a(\mathbf{i})=g(\mathbf{i})L(\mathbf{i})$, $g$ is a GHK(B), $L$ is some  asymptotically negligible function (see (\ref{eq:LRD a}) and the lines below), and the prime $'$ means that we do not sum on the diagonals $i_p=i_q$, $p\neq q\in \{1,\ldots,k\}$, i.e., the summation in (\ref{eq:off diagonal}) is only over unequal $i_1,\ldots,i_k$. We note that when $a(\cdot)$ is symmetric,  the autocovariance of $X'(n)$ in (\ref{eq:off diagonal}) is
$$
\gamma(n)=\E X'(n)X'(0)= k! \sum_{\mathbf{i}\in \mathbb{Z}_+^k}'a(\mathbf{i})a(\mathbf{i}+n\mathbf{1}),~ n\ge 0.
$$
\begin{Rem}
The difference between the \emph{discrete chaos process} $X'(n)$ defined in (\ref{eq:off diagonal}) and the Volterra process $X(n)$ in (\ref{eq:X(n)})  is the exclusion of the diagonals.
\end{Rem}

\section{$L^2(\Omega)$-definiteness}\label{Sec:inf homo sum}
In this section, we derive conditions under which a $k$-th order  polynomial form with diagonals is well-defined.

The $k$-th order Volterra process in (\ref{eq:X(n)}) is a  polynomial form in i.i.d.\ random variables $\{\epsilon_i\}$. To allow for long memory and obtain  non-central limit theorems, the coefficient $a(\mathbf{i})$ in (\ref{eq:X(n)}) must be nonzero at an infinite number of $\mathbf{i}\in \mathbb{Z}_+^k$. Otherwise $X(n)$ is an $m$-dependent sequence and thus subject to the central limit theorem (\citet{billingsley:1956:invariance}). So the first problem is to ensure that such a polynomial form with an infinite number of terms is well-defined, that is, to determine when the following random variable is well-defined:
\begin{equation}\label{eq:diag poly}
X=\sum_{0<i_1,\ldots,i_k<\infty}a(i_1,\ldots,i_k)\epsilon_{i_1}\ldots \epsilon_{i_k}=\sum_{\mathbf{i}\in \mathbb{Z}^k_+} a(\mathbf{i}) \epsilon_{i_1}\ldots \epsilon_{i_k},
\end{equation}
where $\{\epsilon_i\}$ is an i.i.d.\ sequence such that
\begin{equation}\label{eq:mean-k}
\E \epsilon_i=0,~\E \epsilon_i^2=1, ~\E |\epsilon_i|^{k}<\infty.
\end{equation}
 One can restrict $a(\mathbf{i})$ to be a symmetric function in $\mathbf{i}$, since a permutation of the variables does not affect $X$, but we shall not do so unless indicated, because it is easier to write down non-symmetric $a(\cdot)$'s.

First, we have the following straightforward criterion for the $L^1(\Omega)$-well-definedness of $X$:
\begin{Pro}\label{Pro:L^1 well-defined}
If $\sum_{\mathbf{i}\in \mathbb{Z}^k_+} |a(\mathbf{i})|<\infty $, then $X$ in (\ref{eq:diag poly}) is well-defined in the $L^1(\Omega)$-sense.
\end{Pro}
\begin{proof}
Let 
$$
X_m=\sum_{\mathbf{0}< \mathbf{i}\le m\mathbf{1}} a(i_1,\ldots,i_k)\epsilon_{i_1}\ldots \epsilon_{i_k},\quad m>0.
 $$
 It suffices to check that $X_m$ is a Cauchy sequence in $L^1(\Omega)$. This is true since for any $n>m>0$,
\begin{align*}
\E |X_m-X_n|\le \sum_{m\mathbf{1}<\mathbf{i}\le n\mathbf{1}}|a(\mathbf{i})| ~\E|\epsilon_{i_1}\ldots\epsilon_{i_k}|\le C \sum_{m\mathbf{1}<\mathbf{i}\le n\mathbf{1}}|a(\mathbf{i})|,
\end{align*}
where $\E|\epsilon_{i_1}\ldots\epsilon_{i_k}|$ is bounded above by a constant because of the assumption $\E|\epsilon_i|^k<\infty$ in (\ref{eq:mean-k}).
\end{proof}

The absolute summability assumption  in Proposition \ref{Pro:L^1 well-defined} is easy to work with, but it is unfortunately too restrictive for incorporating long memory.
 We will introduce instead a condition on $a(\mathbf{i})$ so that $X$ is well-defined in the $L^2(\Omega)$-sense. Beside the obvious assumption $\E \epsilon_i^{2k}<\infty$, some delicate assumptions on $a(\mathbf{i})$  need to be imposed, which are stated in Proposition \ref{Pro:diag well def} below.
We first give an outline of the idea. If $X$ in (\ref{eq:diag poly}) is instead defined as an off-diagonal polynomial form:
\begin{equation}\label{eq:off diag poly}
X'=\sum_{\mathbf{i}\in \mathbb{Z}^k_+}' a(\mathbf{i}) \epsilon_{i_1}\ldots \epsilon_{i_k},
\end{equation}
then due to the off-diagonality, it is easy to see that the $L^2(\Omega)$-well-definedness of $X'$ is guaranteed by the simple square-summability condition:
$$\sum_{\mathbf{i}\in \mathbb{Z}^k_+}' a(\mathbf{i})^2<\infty,$$
which equals $(k!)^{-1}\E X'^2$ if $a(\cdot)$ is symmetric. In fact, this $L^2(\Omega)$-defineness criterion still holds if one has more generally
\begin{equation}\label{eq:off diagonal poly different noise}
X'=\sum_{\mathbf{i}\in \mathbb{Z}^k_+}' a(\mathbf{i}) \epsilon_{i_1}^{(1)}\ldots \epsilon_{i_k}^{(k)},
\end{equation}
where $\{\mathbd{\epsilon}_{i}:=(\epsilon_i^{(1)},\ldots, \epsilon_{i}^{(k)}),i\in \mathbb{Z}\}$ forms an i.i.d.\ sequence of $k$-dimensional vector with mean $0$ and finite variance in each component.   We will need this fact below.

In order to check that  the polynomial-form  in (\ref{eq:diag poly}), which includes diagonals, is well-defined, we shall decompose it into a finite number of off-diagonal polynomial forms, and check the well-definedness of each using the simple square-summability condition.
In order to do this, we introduce  some further notation, which will also be useful in the sequel.

We let $\mathcal{P}_k$  denote all the partitions of $\{1,\ldots,k\}$. If $\pi \in \mathcal{P}_k$, then $|\pi|$ denotes the number of sets in the partition. If we have a variable $\mathbf{i}\in\mathbb{Z}^k_+$ , then $\mathbf{i}_\pi$ denotes a new variable where its components are identified according to  $\pi$. For example, if $k=3$, $\pi=\{\{1,2\},\{3\}\}$ and $\mathbf{i}=(i_1,i_2,i_3)$, then $\mathbf{i}_\pi=(i_1,i_1,i_2)$. In this case we write $\pi=\{P_1,P_2\}$ where $P_1=\{1,2\}$ and $P_2=\{3\}$. If $a(\cdot)$ is a function on $\mathbb{Z}^k_+$, then $a_\pi(i_1,\ldots,i_m):=a(\mathbf{i}_\pi)$, where $m=|\pi|$. In the preceding example, $a_{\pi}(\mathbf{i})=a(i_1,i_1,i_2)$ with $m=2$.

Suppose that $\pi=\{P_1,\ldots, P_{|\pi|}\}$, where $P_i\cap P_j=\emptyset$, $\cup_i P_i=\{1,\ldots,k\}$. We  suppose throughout that the $P_i$'s are ordered according to their smallest element. In the preceding example, $P_1=\{1,2\}$ and $P_2=\{3\}$. We define the following summation  operation on a function $a(\cdot)$ on $\mathbb{Z}^k_+$.
\begin{Def}\label{Def:Sum Notation}
For any $T\subset \{1,\ldots,|\pi|\}$, the summation $S'_T(a_\pi)$ is obtained by summing $a_\pi$ over its variables indicated by $T$ off-diagonally, yielding a function with $|\pi|-|T|$ variables.
\end{Def}
For instance, if $\pi=\{\{1,5\},\{2\},\{3,4\}\}$, then $\mathbf{i}_\pi=(i_1,i_2,i_3,i_3,i_1)$  and if $T=\{1,3\}$, then
\begin{equation}\label{eq:S eg}
(S'_T a_\pi)(i)=\sum_{i_1,i_3}' a(i_1,i,i_3,i_3,i_1),
\end{equation}
provided that it is well-defined. Note that in this off-diagonal sum,  we require, in addition to  $i_1\neq i_3$, that neither $i_1$ nor $i_3$  equals to $i$.
 If $T=\emptyset$, $S'_T$ is understood to be the identity operator, where no summation is performed.

We need also Appell polynomials  which we briefly introduce here. For more details, see, e.g. \citet{avram:1987:noncentral} or Chapter 3.3 of \citet{beran:2013:long}.  Given a random variable $\epsilon$ with $\E|\epsilon|^K<\infty$,  the $k$-th order Appell polynomial with respect to the law of $\epsilon$, is defined through the following recursive relation:
\begin{align*}
\frac{d}{dx}A_p(x)= pA_{p-1}(x),\quad \E A_p(\epsilon)=0, \quad  A_0(x)=1, \quad p=1,\ldots,K.
\end{align*}
For example, if $\mu_p=\E \epsilon^p$, then $A_1(x)=x-\mu_1$, $A_2(x)=x^2-2\mu_1 x+2\mu_1^2-\mu_2$, etc. If in addition $\mu_1=0$, then $A_1(x)=x$, and $A_2(x)=x^2-\mu_2$. For consistency, one sets $\mu_0=\E \epsilon^0=1$. We will use
an important property of Appell polynomials, namely,  for any integer $p\ge 0$,
\begin{equation}\label{eq:appell decomp}
x^p = \sum_{j=0}^p {p\choose j} \mu_{p-j}A_j(x).
\end{equation}

\begin{Pro}\label{Pro:diag well def}
The polynomial form $X$ in (\ref{eq:diag poly}) is a random variable defined in the $L^2(\Omega)$-sense,
if the following three conditions hold:
\begin{enumerate}
\item $\E \epsilon_i^{2k}<\infty$;
\item $a(\cdot)$ satisfies the following: for any $\pi=\{P_1,\ldots,P_{|\pi|}\} \in \mathcal{P}_k$, we have
\begin{equation}\label{eq:diag 1 cond}
\sum_{0<i_1,\ldots,i_{|\pi|}<\infty}' a_\pi(i_1,\ldots,i_{|\pi|})^2<\infty;
\end{equation}
\item for any $\pi\in \mathcal{P}_k$ and any nonempty
$T\subset \{1,\ldots,|\pi|\}$ satisfying $|P_t|\ge 2$ for all $t\in T$,  we have
\begin{equation}\label{eq:diag 2 cond}
\sum'_{0<i_1,\ldots,i_{|\pi|-|T|}<\infty} \left[(S'_T |a_\pi|)(i_1,\ldots,i_{|\pi|-|T|})\right]^2<\infty,
\end{equation}
where if $|T|=|\pi|$, (\ref{eq:diag 2 cond}) is understood as merely stating that the sum $S'_T |a_\pi|$ converges.
\end{enumerate}
\end{Pro}
\begin{Rem}
To understand the need for (\ref{eq:diag 1 cond}) and (\ref{eq:diag 2 cond}), note that, in order to use the $L^2(\Omega)$-definiteness of (\ref{eq:off diagonal poly different noise}), it is necessary to center the powers of $\epsilon_i$. For example, consider
 $$
X=\sum_{i_1,i_2,i_3>0}a(i_1,i_2,i_3)\epsilon_{i_1}\epsilon_{i_2}\epsilon_{i_3}.
 $$
 If we focus on the subset $\{i_1=i_2\neq i_3\}$, then we have
\begin{align*}
\sum_{i_1,i_2>0}'a(i_1,i_1,i_2)\epsilon_{i_1}^2 \epsilon_{i_2}&=\sum_{i_1,i_2>0}'a(i_1,i_1,i_2)(\epsilon_{i_1}^2-\mu_2)\epsilon_{i_2}+\mu_2\sum_{i_2>0}\sum_{i_1\neq i_2>0}a(i_1,i_1,i_2)\epsilon_{i_2}
\\
&=\sum_{i_1,i_2>0}'a(i_1,i_1,i_2)A_2(\epsilon_{i_1})A_1(\epsilon_{i_2})+\mu_2\sum_{i_2>0}\sum_{i_1\neq i_2>0}a(i_1,i_1,i_2)A_1(\epsilon_{i_2}),
\end{align*}
where $\mu_2=\E\epsilon_i^2$. For the preceding two terms to be well-defined in $L^2(\Omega)$, we require
respectively
\begin{align*}
\sum_{i_1,i_2>0}'a(i_1,i_1,i_2)^2<\infty
\end{align*}
and
\[
\sum_{i_2>0}\left(\sum_{i_1\neq i_2>0}a(i_1,i_1,i_2)\right)^2
=\sum_{i_2>0} \left[(S'_Ta_\pi)(i_2)\right]^2<\infty, ~\pi=\{\{1,2\},\{3\}\},~T=\{1\}.
\]
\end{Rem}
An example of $a(\cdot)$ satisfying (\ref{eq:diag 1 cond})  but not (\ref{eq:diag 2 cond}) is given by:
\[
a(i_1,i_2)=(i_1+i_2)^{-1}(\log i_2)^{-1}.
\]
Note that $a(i_1,i_2)^2=(i_1+i_2)^{-2}(\log i_2)^{-2}$ is summable because $\sum_{i_2=2}^\infty i_2^{-1}(\log i_2)^{-2}$ is finite by the integral test, while $a(i,i)=\frac{1}{2} i^{-1}(\log i)^{-1}$ is not summable.

\begin{proof}[Proof of Proposition \ref{Pro:diag well def}]
By collecting various diagonal cases, we express $X$ as
\begin{equation}\label{eq:X dec1}
X=\sum_{\pi\in \mathcal{P}_k}\sum_{0<i_1,\ldots,i_{m}<\infty}' a_\pi(i_1,\ldots,i_{m})\epsilon_{i_1}^{p_1}\ldots \epsilon_{i_{m}}^{p_m},
\end{equation}
where $\pi=\{P_1,\ldots,P_{|\pi|}\} \in \mathcal{P}_k$, $m=|\pi|$, $p_j=|P_j|\ge 1$, $j=1,\ldots,m$, $p_1+\ldots+ p_m=k$.
Since $\mathcal{P}_k$ is finite, one can focus on the $L^2(\Omega)$-definedness of each term
\[
X_\pi: =\sum_{0<i_1,\ldots,i_m<\infty}' a_\pi(i_1,\ldots,i_m)\epsilon_{i_1}^{p_1}\ldots \epsilon_{i_m}^{p_m}.
\]

Let $A_j(x)$ be the $j$-th order Appell polynomial with respect to the law of $\epsilon_i$. Let 
$$
\mu_j=\E \epsilon_i^j.
$$
 Then by (\ref{eq:appell decomp}),
\begin{align*}
\epsilon_{i_1}^{p_1}\ldots \epsilon_{i_m}^{p_m}&= \sum_{j_1=0}^{p_1}\ldots \sum_{j_m=0}^{p_m} {p_1\choose j_1}\ldots {p_m\choose j_m} \mu_{p_1-j_1}\ldots \mu_{p_m-j_m} A_{j_1}(\epsilon_{i_1})\ldots A_{j_m}(\epsilon_{i_m}).
\end{align*}
Thus to ensure $X_\pi \in L^2(\Omega)$, it suffices to show that
\begin{equation}\label{eq:X_pi(j)}
X_\pi^{\mathbf{j}}:=\sum_{0<i_1,\ldots,i_m<\infty}'a_\pi(i_1,\ldots,i_m) {p_1\choose j_1}\ldots {p_m\choose j_m} \mu_{p_1-j_1}\ldots \mu_{p_m-j_m}  A_{j_1}(\epsilon_{i_1})\ldots A_{j_m}(\epsilon_{i_m})
\end{equation}
is well-defined in $L^2(\Omega)$ for any $(j_1,\ldots,j_m)\in\{0,1,\ldots,p_1\}\times \ldots \times \{0,1,\ldots,p_m\}$.

Note now the following crucial fact.
 Since $\mu_1=\E\epsilon_i=0$ by assumption, we do not need to consider $p_t-j_t=1$ in (\ref{eq:X_pi(j)}). Thus:
\begin{equation}\label{eq:j_t=0}
\text{If $j_t=0$, then we need to consider only $p_t=|P_t|\ge 2$.}
\end{equation}

Suppose first that $j_1,\ldots,j_m\ge 1$.  Since by assumption $\E A_j(\epsilon_i)=0$ and $\E A_j(\epsilon_i)^2<\infty$ for $1\le j\le k$, then in view of the discussion concerning (\ref{eq:off diagonal poly different noise}), it is sufficient to require (\ref{eq:diag 1 cond}). Now suppose that some $j_t=0$, and observe that  $A_{j_t}(\epsilon_i)$ is then the constant $1$. Thus if $T$ is the set of $t$'s such that $j_t=0$, then
\begin{equation}\label{eq:Xpi-j}
X_\pi^{\mathbf{j}}=\sum_{0<i_{t_1},\ldots,i_{t_{m-r}}<\infty}' (S'_Ta_\pi)(i_{t_1},\ldots,i_{t_{m-r}})c(\mathbf{p},\mathbf{j}) A_{j_{t_1}}(\epsilon_{i_{t_1}})\ldots A_{j_{t_{m-r}}}(\epsilon_{i_{t_{m-r}}}),
\end{equation}
where  $\{t_1,\ldots,t_{m-r}\}=\{1,\ldots,m\}\setminus T$, $r=|T|$ and
\begin{equation}\label{eq:cpj}
c(\mathbf{p},\mathbf{j})={p_1\choose j_1}\ldots {p_m\choose j_m} \mu_{p_1-j_1}\ldots \mu_{p_m-j_m}.
\end{equation}
 So one can bound $\E (X_{\pi}^{\mathbf{j}})^2$ by a constant times the sum in (\ref{eq:diag 2 cond}) since (\ref{eq:Xpi-j}) has the form (\ref{eq:off diagonal poly different noise}).

\end{proof}
\begin{Rem}
Since $\E A_j(\epsilon_i)=0$ for $j\ge 1$, one can see from (\ref{eq:X_pi(j)}) that $\E X_\pi^{\mathbf{j}}\neq 0$  only when $j_1=\ldots=j_m=0$, which implies
\begin{equation}\label{eq:expectation}
\E X=\sum_{\pi\in \mathcal{P}_k}\sum_{\mathbf{i}\in \mathbb{Z}^m_+}'a_{\pi}(\mathbf{i})\mu_{p_1}\ldots \mu_{p_m}.
\end{equation}
Relation (\ref{eq:diag 2 cond}) with $|T|=|\pi|$ ensures that $\E |X|<\infty$.
\end{Rem}

We now state here a practical sufficient condition for Proposition \ref{Pro:diag well def}:
\begin{Pro}\label{Pro:volterra (B) well-defined}
Let $a(\cdot)$ be a function on $\mathbb{Z}_+^k$ such that
\[
|a(\mathbf{i})|\le c \prod_{j=1}^k i_j^{\gamma_j},
\]
where $c>0$ is some constant and $\gamma_j<-1/2$, $j=1,\ldots,k$.
Then
\[
X:=\sum_{\mathbf{i}\in \mathbb{Z}_+^k}a(\mathbf{i}) \epsilon_{i_1}\ldots \epsilon_{i_k}
\]
is a well-defined random variable in $L^2(\Omega)$, where $\{\epsilon_i\}$ is i.i.d.\ with mean $0$ and variance $1$ and $\E |\epsilon_i|^{2k}<\infty$.
\end{Pro}
\begin{proof}

We set $m=|\pi|$. Relation (\ref{eq:diag 1 cond}) holds because
\begin{equation}\label{eq:a_pi bound proof}
|a_\pi(\mathbf{i})|\le c_1 \prod_{j=1}^m i_j^{\beta_j}
\end{equation}
for some $c_1>0$, where $\beta_j\le \gamma_j<-1/2$, so
\[
\sum_{0<i_1,\ldots,i_m<\infty}'a_\pi(i_1,\ldots,i_m)^2 \le c_1^2\sum_{0<i_1,\ldots,i_m<\infty} i_1^{2\beta_1}\ldots i_m^{2\beta_m}<\infty.
\]
To check (\ref{eq:diag 2 cond}), note that when $t\in T$, we have $|P_t|\ge 2$ by (\ref{eq:j_t=0}), and so we have in addition $\beta_{j_t}\le 2\gamma_{j_t}<-1$ in (\ref{eq:a_pi bound proof}). Thus for some finite $c_2,c_3>0$,
\[
\Big[S'_T |a_\pi| (\mathbf{i})\Big]^2\le  c_2 \left[\sum_{i_t, t\in T}\left( \prod_{j_t,t\in T}i_{j_t}^{\beta_{j_t}}\right) \right]^2 \left( \prod_{j_s,s\notin T}i_{j_s}^{\beta_{j_s}}\right)^2= c_3 \left( \prod_{j_s,s\notin T}i_{j_s}^{2\beta_{j_s}}\right),
\]
where the summation in the middle is finite, and hence
\[
\sum_{0<i_{j_s}<\infty,s\notin T} \prod_{j_s, s\notin T} i_{j_s}^{2\beta_s}<\infty.
\]
\end{proof}

\section{Volterra processes with long memory}\label{Sec:LRD Volterra}
We introduce in this section the $k$-th order Volterra processes for which we establish non-central limit theorems in Section \ref{Sec:NCLT}.
\subsection{The off-diagonal process}
 We first introduce for convenience  the following $k$-th order discrete chaos process with different noises:
\begin{equation}\label{eq:off-diagonal chaos general}
X'(n):=\sum_{\mathbf{i}\in\mathbb{Z}_+^k}'a(\mathbf{i})\epsilon_{n-i_1}^{(1)}\ldots\epsilon_{n-i_k}^{(k)},
\end{equation}
where $\{\mathbd{\epsilon}_i:=(\epsilon_i^{(1)},\ldots, \epsilon_{i}^{(k)}), i\in \mathbb{Z}\}$ is an i.i.d.\ sequence of vectors, where $\E \epsilon^{(p)}_i=0$  and $\E|\epsilon^{(p)}_i|^2<\infty$, $p=1,\ldots,k$.
This is just an extension of  (\ref{eq:off diagonal}) adapted to (\ref{eq:off diagonal poly different noise}). For such $X'(n)$, it is easy to show that the autocovariance satisfies
\begin{align}\label{eq:acf bound}
|\gamma(n)|\le c k!\sum_{\mathbf{i}\in \mathbb{Z}_+^k}' \widetilde{|a|}(\mathbf{i})\widetilde{|a|}(\mathbf{i}+n\mathbf{1}), ~n\ge 0,
\end{align}
where $\widetilde{|a|}$ denotes the symmetrization of the absolute value $|a|(\mathbf{i}):=|a(\mathbf{i})|$, and $c>0$ is a constant which accounts for the covariance between different components of $\mathbd{\epsilon}_{i}$. For example, suppose $X'(n)=\sum_{i_1,i_2>0}' a(i_1,i_2)\epsilon_{i_1}^{(1)}\epsilon_{i_2}^{(2)}$, and $\sigma(p,q)=\E \epsilon_i^{(p)}\epsilon_i^{(q)}$, then for $n>0$,
\begin{align*}
\left|\E X'(n)X'(0)\right|&=\left|\sum_{i_1,i_2>0}' a(i_1+n,i_2+n)a(i_1,i_2) \sigma(1,1)\sigma(2,2)+\sum_{i_1,i_2>0}' a(i_1+n,i_2+n)a(i_2,i_1) \sigma(1,2)^2\right|\\
&\le C \Big(\sum_{i_1,i_2>0}' \big(|a(i_1,i_2)|+|a(i_2,i_1)|\big)\big(|a(i_1+n,i_2+n)|+|a(i_2+n,i_1+n)|\big)\Big)
\end{align*}
for some constant $C>0$.
\subsection{Off-diagonal decomposition of the Volterra process}
We will focus on  the $k$-th order Volterra process $X(n)$ in (\ref{eq:X(n)})   with  coefficients given as
\begin{equation}\label{eq:LRD a}
a(\mathbf{i})=g(\mathbf{i})L(\mathbf{i}),
\end{equation}
where $g$ is a GHK(B) on $\mathbb{R}_+^k$ with homogeneity exponent $\alpha \in (-\frac{k+1}{2},-\frac{k}{2})$ (see Definition \ref{Def:class bounded}), and $L$ is a bounded real-valued function on $\mathbb{Z}_+^k$ such that  for any $\mathbf{x}\in \mathbb{R}_+^k$ and any bounded $\mathbb{Z}_+^k$-valued function $\mathbf{B}(n)$, we have
\begin{equation}\label{eq:L assump}
\lim_{n\rightarrow \infty}L([n\mathbf{x}]+\mathbf{B}(n))=1
\end{equation}
(see \citet{bai:taqqu:2013:generalized} equation (25) and Remark 4.5).

\begin{Pro}\label{Pro:X(n) well defined}
The process $X(n)$
is  well-defined in the $L^2(\Omega)$-sense.
\end{Pro}
\begin{proof}
Follows from Remark \ref{Rem:bound prod} and Proposition \ref{Pro:volterra (B) well-defined}.
\end{proof}

The off-diagonal decomposition (\ref{eq:X dec1}) of a homogeneous polynomial form obtained in the proof of Proposition \ref{Pro:diag well def} plays  also a crucial role in analyzing the autocovariance  of $X(n)$ and deriving limit theorems. As in (\ref{eq:X dec1}) and (\ref{eq:X_pi(j)}), we have
\begin{equation*}
X(n)=\sum_{\pi\in \mathcal{P}_k}\sum_{\mathbf{j}\in J_{\pi}}X_{\pi}^\mathbf{j}(n),
\end{equation*}
where $\mathcal{P}_k$ is the set of all partitions of $\{1,\ldots,k\}$,  $\pi=\{P_1,\ldots,P_m\}$, $p_t=|P_t|$, $J_{\pi}=\{0,1,\ldots,p_1\}\times \ldots \times \{0,1,\ldots,p_m\}$,
\begin{equation}\label{eq:X_pi^j(n)}
X_{\pi}^\mathbf{j}(n):=\sum_{\mathbf{i}\in \mathbb{Z}_+^m}' a_\pi(\mathbf{i})c(\mathbf{p},\mathbf{j}) A_{j_1}(\epsilon_{n-i_1})\ldots A_{j_m}(\epsilon_{n-i_m}),
\end{equation}
with $c(\mathbf{p},\mathbf{j})$ given as in (\ref{eq:cpj}). Note that $X_{\pi}^\mathbf{j}(n)$ is of the form (\ref{eq:off-diagonal chaos general}), where $a_\pi(\mathbf{i})c(\mathbf{p},\mathbf{j})$ replaces $a(\mathbf{i})$ and where $ A_{j_1}(\epsilon_{n-i_1}),\ldots, A_{j_m}(\epsilon_{n-i_m})$ are independent random variables replacing $\epsilon_{n-i_1}^{(1)},\ldots ,\epsilon_{n-i_k}^{(k)}$ with $m$ playing the role of  $k$. In view of (\ref{eq:expectation}), we have
\begin{equation}\label{eq:X_c}
X_c(n):=X(n)-\E X(n)=\sum_{\pi\in \mathcal{P}_k}\sum_{\mathbf{j}\in J_{\pi}^+}X_{\pi}^\mathbf{j}(n),
\end{equation}
with
\begin{equation}\label{eq:Jpi+}
J_{\pi}^+=\{0,1,\ldots,p_1\}\times \ldots \times \{0,1,\ldots,p_m\}\setminus \{(0,0,\ldots,0)\}
\end{equation}
instead. We recall again that since $\mu_1=0$, whenever  $j_t=0$, we need  to consider only $p_t\ge 2$, $t=1,\ldots,m$. Thus $X_\pi^{\mathbf{j}}$ can be further expressed as (\ref{eq:Xpi-j}). Note that while $m$ denotes the number of Appell polynomials in the product (\ref{eq:X_pi^j(n)}), $m-r$ denotes the number of Appell polynomials in the product (\ref{eq:Xpi-j}) where  each Appell polynomial has a positive order, those of order $0$ having been incorporated in $S'_T$.

Our first step is to obtain the asymptotic behavior of the \emph{autocovariance} of $X(n)$ or $X_c(n)$ when $g$ is a GHK(B). To this end we need some intermediate  results. We will repeatedly use the following elementary asymptotics:
if $\gamma<-1$, then
\begin{equation}\label{eq:sum power +1}
\sum_{n=1}^\infty (l+n)^\gamma =\sum_{n=l+1}^\infty n^\gamma \sim -(\gamma+1)^{-1} l^{\gamma+1}, \text{ as $l\rightarrow+\infty$}.
\end{equation}
A parallel result but with  equality holds for integration:
\begin{equation}\label{eq:int power +1}
\int_{0}^\infty (x+y)^\gamma dx = -(\gamma+1)^{-1}y^{\gamma+1}.
\end{equation}
Relation (\ref{eq:sum power +1}) can be derived using (\ref{eq:int power +1}) and an  integral approximation argument.
\begin{Lem}\label{Lem:int g}
Suppose that $g$ is a  GHK(B) of order $k$ with homogeneity exponent $\alpha\in \left(-\frac{k+1}{2},-\frac{k}{2}\right)$.  Let $0\le r<m=k-r$, then
\begin{equation}\label{eq:g_r}
g_r(\mathbf{x}):=\int_{\mathbb{R}_+^{r}}g(y_1,y_1,\ldots,y_r,y_r,x_{1},\ldots,x_{m-r}) dy_1\ldots dy_r
\end{equation}
 is a GHK on $\mathbb{R}_+^{k_r}$ with $k_r=k-2r$ and homogeneity exponent $\alpha_r=\alpha+r\in \left(-\frac{k_r+1}{2},-\frac{k_r}{2}\right)$.
\end{Lem}
\begin{proof}
$g_r$ is well-defined, since by  Definition \ref{Def:class bounded} of GHK(B),  for some constant $C>0$, we have
\begin{equation}\label{eq:g_r bound}
|g(y_1,y_1,\ldots,y_r,y_r,x_{1},\ldots,x_{m-r})|\le C (y_1+\ldots+y_r+x_1+\ldots+x_{m-r})^\alpha.
\end{equation}
Thus by  applying (\ref{eq:int power +1}) iteratively, we need only to note that $\alpha+r<0$, because $2r<k$  and $\alpha<-k/2$.
We now check Condition \ref{ass:homo} of Definition \ref{Def:GHK}, that is,  the homogeneity of $g_r(\cdot)$. We have for any $\lambda>0$ that
\begin{align*}
g_r(\lambda\mathbf{x})&=\int_{\mathbb{R}_+^{r}}g(y_1,y_1,\ldots,y_r,y_r, \lambda x_{1},\ldots,\lambda x_{m-r}) dy_1\ldots dy_r\\
&=\int_{\mathbb{R}_+^r} g(\lambda y_1,\lambda y_1,\ldots,\lambda y_r,\lambda y_r,\lambda x_1,\ldots,\lambda x_{m-r})d(\lambda y_1)\ldots d(\lambda y_r)
\\
&=\lambda^{\alpha+r} \int_{\mathbb{R}_+^{r}}g(y_1,y_1,\ldots, y_r,y_r,  x_{1},\ldots, x_{m-r}) dy_1\ldots dy_r=\lambda^{\alpha+r}g_r(\mathbf{x}).
\end{align*}
We check then Condition \ref{ass:int2} of Definition \ref{Def:GHK}.
Integrating both sides of (\ref{eq:g_r bound}) with respect to $(y_1,\ldots,y_r)\in \mathbb{R}_+^r$ shows $|g_r(\mathbf{x})|\le c\|\mathbf{x}\|^{\alpha+r}$ for some $c>0$. So Condition \ref{ass:int2} in Definition \ref{Def:GHK} is satisfied in view of  Remark \ref{Rem:bound prod}.

\end{proof}
\begin{Rem}\label{Rem:H}
The index $r$ in  $g_r$ refers to the number of pairs of variables in $g$ that are identified. The number $m$ denotes the number of different variables in (\ref{eq:g_r}), and the number $k=2r+(m-r)=m+r$ denotes the total number of variables in $g$. Finally, $k_r=k-2r=m-r$ indicates the number of $x$ variables, that is, the size of the argument of $g_r$.  All the GHK(B) $g_r$, $r=0,\ldots,[k/2]$, obtained in (\ref{eq:g_r}), have the same  $H=\alpha_r+k_r/2+1$ (homogeneity exponent+dimension/2+1). This is because $\alpha_r=\alpha+r$ and $k_r=k-2r$.
\end{Rem}
\begin{Rem}\label{Rem:variable left}
Lemma \ref{Lem:int g} assumes $r<m$ or equivalently $k_r=k-2r>0$. In other words, that there is  a positive number of $x$ variables of $g_r(\cdot)$ in (\ref{eq:g_r}).
\end{Rem}

\subsection{Behavior of the autocovariances}

We have the following asymptotics for the autocovariance of $X_{\pi}^{\mathbf{j}}(n)$, which are the off-diagonal terms of $X_c(n)$ in (\ref{eq:X_c}). Note that  $\mathbf{j}=(j_1,\ldots,j_p)\neq\mathbf{0}$ because of centering, so $\sum_{t=1}^m \mathrm{1}_{\{j_t=0\}}<m$. Recall that by assumption we have $-1<2\alpha+k<0$.
\begin{Pro}\label{Pro:basic acf}
Let $m=|\pi|$, and $r=\sum_{t=1}^m \mathrm{1}_{\{j_t=0\}}<m$.
\begin{enumerate}[(i)]
\item If $m+r=k$ , then the autocovariance $\gamma(n)$ of $X_{\pi}^{\mathbf{j}}(n)$ satisfies
\begin{equation}\label{eq:lrdj}
\gamma(n)\sim c n^{2\alpha+k}
\end{equation}
as $n\rightarrow \infty$,
for some constant $c>0$.
\item If $m+r<k$, then
\begin{equation}\label{eq:srdj}
\sum_n |\gamma(n)|<\infty.
\end{equation}
\end{enumerate}
\end{Pro}
\begin{proof}
We claim first that $m+r=k$ if and only if in the partition $\pi=\{P_1,\ldots,P_m\}$, every $|P_t|\le 2$, and whenever $|P_t|=2$, one has $j_t=0$. Indeed, as noted in (\ref{eq:j_t=0}), if $j_t=0$, then $|P_t|\ge 2$, and thus
$$k=\sum_{t=1}^m 1_{\{j_t=0\}}|P_t|  +\sum_{t=1}^m 1_{\{j_t>0\}}|P_t|\ge 2r+(m-r)=m+r.$$
The equality is attained only if when $j_t=0$, $|P_t|=2$, and when $j_t>0$, $|P_t|=1$.

Suppose first that $m+r=k$. We can  assume without loss of generality in (\ref{eq:LRD a}) that $a(\cdot)=g(\cdot)$ is symmetric and $L=1$ (including a general $L$ in the following argument is easy).

Using the symmetry of $g$, $X_\pi^{\mathbf{j}}(n)$ in (\ref{eq:X_pi^j(n)}) simplifies. To compute it, note that since $r$ corresponds to the number of Appell polynomials of order $0$ which are all equal to $1$, we have
\[
j_1=\ldots=j_r=0,~ j_{r+1}=\ldots = j_m=1, ~A_{j_1}(\epsilon)=\ldots=A_{j_r}(\epsilon)=1,~ A_{j_{r+1}}(\epsilon)=\ldots A_{j_m}(\epsilon)=\epsilon,
\]
\[
\pi =\Big\{\{1,2\},\ldots,\{2r-1,2r\},\{2r+1\},\ldots\{m+r\}\Big\},
\]
\[
\mathbf{p}=(\underbrace{2,\ldots,2}_r,\underbrace{1,\ldots,1}_{m-r}),~
c(\mathbf{p},\mathbf{j})={2\choose 0}\ldots {2\choose 0} {1\choose 1}\ldots {1\choose 1}=1,
\]
and $\mu_0=\E \epsilon^0 =1$, $\mu_2=\E \epsilon^2=1$. We therefore get
\begin{equation}\label{eq:basic contri}
X_{\pi}^{\mathbf{j}}(n)=\sum_{(\mathbf{i},\mathbd{\ell})\in \mathbb{R}_+^k}' g(i_1\mathbf{1}_2,\ldots,i_r\mathbf{1}_2,l_{1},\ldots,l_{m-r}) \epsilon_{n-l_{1}}\ldots \epsilon_{n-l_{m-r}},
\end{equation}
where $\mathbf{1}_2$ denotes the vector made of two $1$'s.
 Let 
 $$
 \mathbf{i}_{t}=(i_{1,t}\mathbf{1}_2,\ldots,i_{r,t}\mathbf{1}_2),\quad t=1,2, \quad \mathbd{\ell}=(l_1,\ldots,l_{m-r}). 
 $$
 Since we are excluding the diagonals, let
\[
D(n)=\{(\mathbf{i}_1,\mathbf{i}_2,\mathbd{\ell})>\mathbf{0}:~ i_{u,t}\neq i_{v,t},~ l_u\neq l_v \text{ for} ~u\neq v; \text{ and}~i_{p,1}\neq l_q, ~i_{p,2}\neq l_{q}+n\text{ including the case }p=q\}.
\]
Then
\begin{align}
\gamma(n)
&=(m-r)!\sum_{(\mathbf{i}_1,\mathbf{i}_2,\mathbd{\ell})\in D(n)}g(\mathbf{i}_1,\mathbd{\ell})g(\mathbf{i}_2,\mathbd{\ell}+n\mathbf{1})\notag\\
&=(k-2r)!n^{2\alpha+k}\int_{\mathbb{R}_+^{m-r}}d\mathbf{y}\int_{\mathbb{R}_+^r}d\mathbf{x}_1 \int_{\mathbb{R}_+^r}d\mathbf{x}_2
g\left(\frac{[n\mathbf{x}_1]+\mathbf{1}}{n},\frac{[n\mathbf{y}]+\mathbf{1}}{n}\right) g\left(\frac{[n\mathbf{x}_2]+\mathbf{1}}{n},\mathbf{1}+\frac{[n\mathbf{y}]+\mathbf{1}}{n}\right)\mathrm{1}_{E(n)},\label{eq:autocov pass limit}
\end{align}
since $m+r=k$ implies $m-r=k-2r$ and where $\mathbf{x}_t=(x_{1,t}\mathbf{1}_2,\ldots,x_{r,t}\mathbf{1}_2)$, $t=1,2$, $\mathbf{y}=(y_1,\ldots,y_{m-r})$, and $D(n)$ in the summation is expressed as
\begin{align*}
E(n)=\{&\mathbf{x}_1,\mathbf{x}_2\in \mathbb{R}_+^r, \mathbf{y}\in \mathbb{R}_+^{m-r}: ~[nx_{u,t}]\neq [nx_{v,t}],~ [ny_u]\neq [ny_v]\text{ for }u\neq v; \\&\text{and} ~[nx_{p,1}]\neq [ny_q],~ [nx_{p,2}]\neq [ny_q]+n \text{ including the case }p=q\}.
\end{align*}
Note first that $1_{E(n)}$ converges to $1$ a.e. as $n\rightarrow\infty$.
By Definition \ref{Def:class bounded}, $|g(\mathbf{z})|\le c_0\|\mathbf{z}\|^\alpha$ for some $c_0>0$. Since $\frac{[n\mathbf{x}]+\mathbf{1}}{n}>\mathbf{x}$ and $\alpha<0$, we have
\begin{align*}
&\int_{\mathbb{R}_+^r}
\left|g\left(\frac{[n\mathbf{x}_1]+\mathbf{1}}{n},\frac{[n\mathbf{y}]+\mathbf{1}}{n}\right)\right| d\mathbf{x}_1 \int_{\mathbb{R}_+^r}\left|g\left(\frac{[n\mathbf{x}_2]+\mathbf{1}}{n},\mathbf{1}+\frac{[n\mathbf{y}]+\mathbf{1}}{n}\right)\right|d\mathbf{x}_2\\\le&
\int_{\mathbb{R}_+^r}
g^*\left(\frac{[n\mathbf{x}_1]+\mathbf{1}}{n},\frac{[n\mathbf{y}]+\mathbf{1}}{n}\right) d\mathbf{x}_1 \int_{\mathbb{R}_+^r}g^*\left(\frac{[n\mathbf{x}_2]+\mathbf{1}}{n},\mathbf{1}+\frac{[n\mathbf{y}]+\mathbf{1}}{n}\right)d\mathbf{x}_2
\\\le&
 \int_{\mathbb{R}_+^r} g^*\left(\mathbf{x}_1,\mathbf{y}\right)d\mathbf{x}_1 \int_{\mathbb{R}_+^r} g^*\left(\mathbf{x}_2,\mathbf{1}+\mathbf{y}\right) d\mathbf{x}_2 =  :g^*_r(\mathbf{y})g^*_r(\mathbf{1}+\mathbf{y}),
\end{align*}
where $g^*(\mathbf{z})=c_1\|\mathbf{z}\|^{\alpha}$ is function decreasing in its every variable, and $g^*_r(\mathbf{y})=c_2\|\mathbf{y}\|^{\alpha+r}$, $\mathbf{y}\in \mathbb{R}_+^{m-r}$. Observe that $g^*_r$ is  a GHK(B) by Definition \ref{Def:class bounded}  on $\mathbb{R}_+^{m-r}$, since $m-r=k-2r$ and
\[
-\frac{k-2r+1}{2}<\alpha+r<-\frac{k-2r}{2}\iff -\frac{k+1}{2}<\alpha<-\frac{k}{2}.
\]
So $g_r^*(\cdot)$ is  a GHK by Remark \ref{Rem:bound prod}, and hence
\[
\int_{\mathbb{R}_+^{m-r}}d\mathbf{y}\int_{\mathbb{R}_+^r}d\mathbf{x}_1 \int_{\mathbb{R}_+^r} d\mathbf{x}_2  g^*(\mathbf{x},\mathbf{y})g^*(\mathbf{x},\mathbf{1}+\mathbf{y}) =\int_{\mathbb{R}_+^{m-r}}  g^*_r(\mathbf{y})g^*_r(\mathbf{1}+\mathbf{y})d\mathbf{y}<\infty.
\]
One can now let $n\rightarrow\infty$ in  (\ref{eq:autocov pass limit}) through the Dominated Convergence Theorem to get
\begin{align*}
\gamma(n)\sim (k-2r)! n^{2\alpha+k}C_{g_r}, \text{ as }n \rightarrow\infty,
\end{align*}
where
\[
C_{g_r}=\int_{\mathbb{R}_+^{m-r}} g_r(\mathbf{x})g_r(\mathbf{1}+\mathbf{x}) d\mathbf{x}>0
\]
 with $g_r$ obtained as in (\ref{eq:g_r}). Since we have assumed (without loss of generality) that $g$ is symmetric,  it does not matter which of the $r$ variables are integrated out. This proves (\ref{eq:lrdj})

Consider now the case $m+r\le k-1$. Again by the assumption of Definition \ref{Def:class bounded} and the boundedness of $L$, $a(\mathbf{i})\le c(i_1+\ldots+i_k)^{\alpha}$ for some $c>0$. Suppose $T=\{t=1,\ldots,m:~ j_t= 0\}$, $\mathbd{\ell}=(l_1,\ldots,l_{m-r})$. Then by applying (\ref{eq:sum power +1}) iteratively, one has for some $C>0$
\[
S'_T|a_\pi(\mathbd{\ell})|\le C (l_1+\ldots+l_{m-r})^{\alpha+r},
\]
where by Definition \ref{Def:GHK},
\begin{equation}\label{eq:beta}
\alpha+r<-\frac{k}{2}+r\le -\frac{m-r}{2}-\frac{1}{2}\le -1,
\end{equation}
since $ m-r\ge 1$ by assumption.

In view of  (\ref{eq:Xpi-j}) and (\ref{eq:acf bound}),
we are left  to show that
\[\sum_{n>0}\sum_{\mathbd{\ell}\in \mathbb{R}_+^{m-l}}'\widetilde{S'_T(|a_\pi|)}(\mathbd{\ell}) \widetilde{S'_T(|a_\pi|)}(\mathbd{\ell}+n\mathbf{1})<\infty.
\]
This can be seen as follows:
\begin{align*}
&\sum_{n=1}^\infty\sum_{0<l_1<\ldots<l_{m-r}<\infty} (l_1+\ldots+l_{m-r})^{\alpha+r}(l_1+\ldots+l_{m-r}+n)_+^{\alpha+r}\\\le C &\sum_{0<l_1<\ldots<l_{m-r}<\infty}(l_1+\ldots+l_{m-r})^{2(\alpha+r)+1}<\infty,
\end{align*}
for some $C>0$,  where we have used  (\ref{eq:sum power +1}) and the fact $2(\alpha+r)+1+m-r< 0$ by (\ref{eq:beta}).
\end{proof}
As a  corollary of Proposition \ref{Pro:basic acf}, we have
\begin{Cor}\label{Cor:full acf}
If $a(\cdot)$ is as given in (\ref{eq:LRD a}), then
\begin{equation}\label{eq:var sum}
\Var\left[ \sum_{n=1}^N X(n)\right]\sim C N^{2H}, \text{ as }n\rightarrow\infty
\end{equation}
for some $C>0$, where $H=\alpha+k/2+1\in (1/2,1)$.
\end{Cor}
\begin{proof}
If $\gamma(n)$ is the autocovariance of a stationary process $Y(n)$, then
\[
\Var\left[\sum_{n=1}^N Y(n)\right]=N \sum_{|n|<N} \gamma(n) -\sum_{|n|<N} |n| \gamma(n).
\]
It is well-known  that  if $\sum_{n}|\gamma(n)|<\infty$, then for some constant $c_1>0$, $\Var\left[\sum_{n=1}^N Y(n)\right]\le c_1 N$; if $\gamma(n)\sim c_2n^{2H-2}$ for $H\in(1/2,1)$ and  some constant $c_2>0$, then   $\Var\left[\sum_{n=1}^N Y(n)\right]\sim c_3 N^{2H}$ for  some constant $c_3>0$. Now apply these to $X_{\pi}^{\mathbf{j}}(n)$ in the  decomposition (\ref{eq:X_c}) to the two cases  $m+r<k$ and $m+r=k$ in Proposition \ref{Pro:basic acf} respectively. The variance of the sum of  $X_\pi^{\mathbf{j}}(n)$ with $m+r=k$ dominates those with $m+r<k$. Note that the off-diagonal polynomial forms $X^{\mathbf{j}}_\pi(n)$'s are uncorrelated if they have different values of $r$'s because then they have different orders. In addition, the exponent in (\ref{eq:lrdj}) is $2\alpha+k=2H-2$,  by (\ref{eq:H})  and by  the definition of $\alpha$, one has $H\in(1/2,1)$. Therefore (\ref{eq:var sum}) holds.
\end{proof}
\begin{Rem}
In view of the preceding proof, when $m+r<k$,  $X^\mathbf{j}_\pi (n)$ has a summable autocovariance, which is the typical definition of  \emph{short memory} or \emph{short-range dependence}, while if $m+r=k$, the autocovariance of $X^\mathbf{j}_\pi (n)$ has a hyperbolic decay with a power in $(-1,0)$, which is the typical definition of \emph{long memory} or \emph{long-range dependence}, with a Hurst exponent $H=\alpha+k/2+1$.
\end{Rem}

\section{Central limit theorems for $k$-th order Volterra processes}\label{Sec:CLT}
We establish in this section a central limit theorem for $X(n)$ in (\ref{eq:X(n)})  using the off-diagonal decomposition  (\ref{eq:X_c}).

We state first a lemma concerning a comparison of moments of the off-diagonal discrete  chaos in (\ref{eq:off diagonal poly different noise}), which will be used later to establish tightness in the space $D[0,1]$ with uniform topology in the central limit theorem.
\begin{Lem}\label{Lem:hypercontract}
Suppose that $\{\mathbd{\epsilon}_{i}:=(\epsilon_i^{(1)},\ldots, \epsilon_{i}^{(k)}),i\in \mathbb{Z}\}$ forms an i.i.d.\ sequence of $k$-dimensional vector with mean $\mathbf{0}$, and $\E |\epsilon_{i}^{(j)}|^{r}<\infty$ for some $r>2$, $j=1,\ldots,k$. Suppose $a(\cdot)$ is a function defined on $\mathbb{Z}_+^k$ satisfying $\sum_{\mathbf{i}\in \mathbb{Z}_+^k}'a(\mathbf{i})^2<\infty$, so that
\begin{equation}\label{eq:different noise}
Y=\sum_{0<i_1<\ldots<i_k<\infty}a(i_1,\ldots,i_k)\epsilon_{i_1}^{(1)}\ldots\epsilon_{i_k}^{(k)}
\end{equation}
is well defined. Then for any $p\in (2,r)$, there exists a constant $C$ which doesn't depend on $a(\cdot)$, such that
\begin{equation}\label{eq:hyper contra}
[\E |Y|^p]^{1/p} \le C [\E |Y|^2]^{1/2}.
\end{equation}
\end{Lem}
 Lemma 4.3 of \citet{krakowiak:szulga:1986:random} yields (\ref{eq:hyper contra})   when $\epsilon_i^{(1)}=\ldots=\epsilon_i^{(k)}=\epsilon_i$ and $a(\mathbf{i})=a(\mathbf{i})1_{\{\mathbf{i}\le n\mathbf{1}\}}$ for some $n\ge k$, and it is extended  straightforwardly to the case $\sum_{0<i_1<\ldots<i_k<\infty} a(i_1,\ldots,i_k)^2<\infty$ in \citet{bai:taqqu:2013:generalized}. The proof, which develops a martingale structure for
 $$\Big\{X_n:=\sum_{0<i_1<\ldots<i_k\le n}a(i_1,\ldots,i_k)\epsilon_{i_1}\ldots\epsilon_{i_k}, n\ge k\Big\}$$
and uses the square function inequality (Theorem 3.2 of \citet{burkholder:1973:distribution}),  needs   to be  modified to allow non-identical components in $\mathbd{\epsilon}_i$ as in the preceding lemma. We include a proof in Section \ref{Sec:Extension hyper} for completeness.

\begin{Thm}\label{Thm:CLT}
Suppose that the  coefficient $a(\cdot)$ defining the Volterra process $X(n)$ in (\ref{eq:X(n)}) satisfies the assumptions in Proposition \ref{Pro:diag well def}. Suppose also that  for any $\pi=\{P_1,\ldots,P_{|\pi|}\} \in \mathcal{P}_k$,
\begin{equation}\label{eq:acf diag 1 cond}
\sum_{n=1}^\infty\sum_{\mathbf{i}\in \mathbb{Z}_+^{|\pi|}}' \widetilde{|a_\pi|}(\mathbf{i})\widetilde{|a_\pi|}(\mathbf{i}+n\mathbf{1})<\infty,
\end{equation}
where $\sim$ stands for symmetrization,
and that for every
$T\subset \{1,\ldots,|\pi|\}$, $|T|<|\pi|$,   satisfying $|P_t|\ge 2$ for all $t\in T$, we have
\begin{equation}\label{eq:acf diag 2 cond}
\sum_{n=1}^\infty\sum'_{\mathbf{i}\in\mathbb{Z}_+^{|\pi|-|T|}} \widetilde{(S'_T |a_\pi|)}(\mathbf{i})\widetilde{(S'_T|a_\pi|)}(\mathbf{i}+n\mathbf{1})<\infty.
\end{equation}
Then if in addition $\sigma^2:=\sum_n\gamma(n):=\sum_{n}\Cov(X(n),X(0))>0$, we have $\sigma^2<\infty$ and
\begin{equation}\label{eq:CLT fdd}
\frac{1}{\sqrt{N}}\sum_{n=1}^{[Nt]} [X(n)-\E X(n)] \ConvFDD \sigma B(t),
\end{equation}
where $B(t)$ is a standard Brownian motion.

If in addition, the noise $\{\epsilon_i\}$ defining $X(n)$ satisfies $\E\epsilon_i^{2k+\delta}<\infty$ for some $\delta>0$, then $\ConvFDD$  in (\ref{eq:CLT fdd}) can be replaced by the weak convergence in $D[0,1]$ with uniform topology.
\end{Thm}
\begin{proof}
In (\ref{eq:X_c}), $X(n)-\E X(n)$ is expressed as a finite sum of off-diagonal terms $X_\pi^{\mathbf{j}}(n)$ given in (\ref{eq:X_pi^j(n)}). This is, however, similar to Theorem 6.14 of \citet{bai:taqqu:2013:generalized} by noting that (\ref{eq:acf diag 1 cond})  and (\ref{eq:acf diag 2 cond})  are essentially the same as the SRD condition in Definition 6.1 of \citet{bai:taqqu:2013:generalized}. The  only difference is the presence of  non-identically distributed noises since here,  Appell polynomials $A_j(\epsilon_i)$ of different orders  are  involved.  This extension is easy to include. We thus omit the details but  mention just the following two points: the relations (\ref{eq:acf diag 1 cond}) and (\ref{eq:acf diag 2 cond}) imply that the auto(cross-)covariances of   $X_\pi^\mathbf{j}(n)$'s are absolutely summable, in particular, $\sigma^2<\infty$. The proof of the convergence in finite-dimensional distributions uses a truncation argument to reduce the   $X_{\pi}^{\mathbf{j}}(n)$'s to  $m$-dependent sequences.  The tightness in $D[0,1]$ can be established with the help of (\ref{eq:hyper contra}).
\end{proof}

We will now state a more practical condition  than (\ref{eq:acf diag 1 cond}) and (\ref{eq:acf diag 2 cond}):
\begin{Pro}\label{Pro:sufficient cond diag CLT}
Relation (\ref{eq:acf diag 1 cond}) and (\ref{eq:acf diag 2 cond}) hold, if
\begin{equation}
|a(i_1,\ldots,i_k)|\le C i_1^{\gamma_1}\ldots i_k^{\gamma_k}
\end{equation}
where  $C>0$ is some constant and each $\gamma_j<-1$.
\end{Pro}
\begin{proof}
It suffices to show (\ref{eq:acf diag 1 cond}) and (\ref{eq:acf diag 2 cond}) for $a(\mathbf{i})=i_1^{\gamma_1}\ldots i_k^{\gamma_k}$, which is easily checked by the separability of the product and that $\sum_{i>0}\sum_{n>0}i^a(i+n)^b<\infty $ for any $a,b<-1$.
\end{proof}

In contrast, if $X'(n)$ is  the  discrete chaos process as defined in (\ref{eq:off-diagonal chaos general}),
the central limit theorem holds for this process under  weaker assumptions, namely, $\gamma_j<-1/2$ and $\sum_{j=1}^k\gamma_j<-k/2-1/2$ instead of $\gamma_j<-1$. Indeed:
\begin{Pro}\label{Pro:SRD sufficient off-diagonal}
Let $X'(n)$ be given as in (\ref{eq:off-diagonal chaos general}), with $a(\cdot)$ satisfying the following:
\begin{equation}\label{eq:a bound}
|a(i_1,\ldots,i_k)|\le C i_1^{\gamma_1}\ldots i_{k}^{\gamma_k},
\end{equation}
where  $C$ is a positive constant and each $\gamma_j<-1/2$, and $\sum_{j=1}^k\gamma_j<-k/2-1/2$. Then the autocovariance $\gamma(n)$ of $X'(n)$ is absolutely summable. If  in addition $\sigma^2:=\sum_{n=-\infty}^\infty\Cov(X'(n),X'(0))>0$, then $X'(n)$ satisfies the central limit theorem (\ref{eq:CLT fdd}). If a moment higher than $2$ of each $\epsilon_i^{(1)},\ldots,\epsilon_{i}^{(k)}$ exists, then (\ref{eq:CLT fdd}) holds with $\ConvFDD$ replaced by weak convergence $\Rightarrow$ in $D[0,1]$.

The above $\ConvFDD$ or $\Rightarrow$ convergence  also holds for a linear combination of different $X'(n)$'s defined using a common i.i.d.\ noise vector $\mathbd{\epsilon}_i$, where the different $X'(n)$'s in the linear combination can have different orders and  involve subvectors of $\mathbd{\epsilon}_i$, provided that each $X'(n)$ satisfies the  conditions mentioned above.
\end{Pro}
\begin{proof}
In view of the relation (\ref{eq:acf bound}) and the extension of Theorem 6.14 in \citet{bai:taqqu:2013:generalized}  mentioned in the proof of Theorem \ref{Thm:CLT}, we only need to show that
$$
\sum_{n=1}^\infty\sum'_{\mathbf{i}>\mathbd{0}} \widetilde{|a|}(\mathbf{i}+n\mathbf{1}) \widetilde{|a|}(\mathbf{i})<\infty.
$$
In view of the bound (\ref{eq:a bound}), this holds if
\begin{equation}\label{eq:show a suff}
\sum_{n=1}^\infty\sum_{\mathbf{i}>\mathbd{0}} (i_1+n)^{\gamma_1}\ldots (i_k+n)^{\gamma_k}i_1^{\gamma_{\sigma(1)}}\ldots i_{k}^{\gamma_{\sigma(k)}}=\sum_{n=1}^\infty r_{1}(n)\ldots r_{k}(n)<\infty,
\end{equation}
where $\sigma$ is any permutation of $\{1,\ldots,k\}$, $r_j(n)=\sum_{i=1}^\infty(i+n)^{\gamma_j}i^{\gamma_{\sigma(j)}}$. Without loss of generality, we may assume that $-1<\gamma_j<-1/2$.  In this case, using the fact $\int_0^\infty x^{a}(1+x)^b dx= B(a+1,-1-a-b)$ for $a,b\in (-1,-1/2)$, where $\mathrm{B}(\cdot,\cdot)$ is the beta function, and  an integral approximation, one gets
$$r_j(n)\sim \mathrm{B}(\gamma_{\sigma(j)}+1,-1-\gamma_j-\gamma_{\sigma(j)}) n^{\gamma_{j}+\gamma_{\sigma(j)}+1}$$
as $n\rightarrow\infty$. But $\sum_{j=1}^k (2\gamma_j+1)<-k-1+k=-1$ by assumption.  So (\ref{eq:show a suff}) holds.
\end{proof}
\begin{Eg}
To illustrate the statement about linear combinations, let $\{\mathbd{\epsilon}_{i}:=(\epsilon_i^{(1)},\ldots, \epsilon_{i}^{(k)}),i\in \mathbb{Z}\}$ be an i.i.d.\ sequence as in Lemma \ref{Lem:hypercontract}. Suppose that $\{j_1,\ldots,j_{k_1}\}$ and $\{l_1,\ldots,l_{k_2}\}$ are two subsets of $\{1,\ldots,k\}$. Then Proposition \ref{Pro:SRD sufficient off-diagonal} applies to $X'_1(n)+X'_2(n)$, where
\begin{align*}
X_1'(n)&=\sum'_{0<i_1,\ldots,i_{k_1}<\infty} a^{(1)}(i_1,\ldots,i_{k_1})\epsilon_{n-i_1}^{(j_1)}\ldots\epsilon_{n-i_{k_1}}^{(j_{k_1})}\\
X_2'(n)&=\sum'_{0<i_1,\ldots,i_{k_2}<\infty} a^{(2)}(i_1,\ldots,i_{k_2})\epsilon_{n-i_1}^{(l_1)}\ldots\epsilon_{n-i_{k_2}}^{(l_{k_2})},
\end{align*}
where $a^{(1)}$ and $a^{(2)}$ satisfy the conditions of Proposition \ref{Pro:SRD sufficient off-diagonal} with $k$ replaced by $k_1$ and $k_2$ respectively.
\end{Eg}

\section{Non-central limit theorem for $k$-th order Volterra processes}\label{Sec:NCLT}
The non-central limit theorem (NCLT) builds on  a result concerning convergence of  a discrete chaos to a Wiener chaos. Let $h$ be a function defined in $\mathbb{Z}^k$ such that $\sum'_{\mathbf{i}\in \mathbb{Z}^k} h(\mathbf{i})^2<\infty$, where $'$ indicates the exclusion of the diagonals $i_p=i_q$, $p\neq q$.
Let $Q_k(h)$ be defined as follows:
\begin{align}\label{eq:Q_k(h)}
Q_k(h)=Q_k(h,\mathbd{\epsilon})=\sum'_{(i_1,\ldots,i_k)\in \mathbb{Z}^k} h(i_1,\ldots,i_k) \epsilon_{i_1}\ldots \epsilon_{i_k}=\sum'_{\mathbf{i}\in \mathbb{Z}^k} h(\mathbf{i})\prod_{p=1}^k \epsilon_{i_p},
\end{align}
where $\epsilon_i$'s are i.i.d.\ noise. It is easy to see that switching the arguments of $h(i_1,\ldots,i_k)$, does not change $Q_k(h)$. So if $\tilde{h}$ is the symmetrization of $h$, then $Q_k(h)=Q_k(\tilde{h})$.

Suppose now that  we have  a sequence of function vectors $\mathbf{h}_n=(h_{1,n},\ldots,h_{j,n})$ where each $h_{j,n}\in L^2(\mathbb{Z}^{k_j})$, $j=1,\ldots,J$.
\begin{Pro}\label{Pro:Poly->Wiener}(Proposition 4.1 of \citet{bai:taqqu:2013:generalized})
Let
 $$
\tilde{h}_{j,n}(\mathbf{x})=n^{k_j/2}h_{j,n}\left([n\mathbf{x}]+\mathbf{c}_j\right), \quad j=1,\ldots,J,
 $$
 where  $\mathbf{c}_j\in \mathbb{Z}^k$. Suppose that there exists $h_j\in L^2(\mathbb{R}^{k_j})$, such that
\begin{equation}\label{eq:htilde h L2 conv}
\|\tilde{h}_{j,n}-h_j\|_{L^2(\mathbb{R}^{k_j})}\rightarrow 0
\end{equation}
as $n\rightarrow\infty$.  Then, as $n\rightarrow\infty$,
\begin{align*}
\mathbf{Q}:=\Big(Q_{k_1}(h_{1,n}),\ldots,Q_{k_J}(h_{J,n})\Big) \ConvD \mathbf{I}:=\Big(I_{k_1}(h_1),\ldots,I_{k_J}(h_J)\Big),
\end{align*}
where the multiple Wiener-It\^o integrals $I_{k_j}(\cdot)$'s are defined using the same Brownian random measure.
\end{Pro}

We are now ready to state the non-central limit theorem.
We  always assume in the sequel that the coefficient $a(\cdot)$ is of the form (\ref{eq:LRD a}) and symmetric,  with $g$  a symmetric GHK(B).
Proposition \ref{Pro:basic acf} and Corollary \ref{Cor:full acf} show that the basic terms $X_\pi^{\mathbf{j}}(n)$ in the decomposition (\ref{eq:X_c}) will either be long-range dependent  or short-range dependent, and the short-range dependent ones will vanish if the normalization $N^{-H}$ used for long-range dependent terms is applied.
\begin{Thm}\label{Thm:Basic NCLT}
Let $X(n)$ be a $k$-th order Volterra process given  in (\ref{eq:X(n)}), with the coefficient $a(\mathbf{i})=g(\mathbf{i})L(\mathbf{i})$ given as in (\ref{eq:LRD a}), where $g$ is a symmetric GHK(B) on $\mathbb{R}_+^k$ with homogeneity exponent $\alpha\in (-k/2-1/2,-k/2)$. Then one has the following weak convergence in $D[0,1]$:
\begin{equation}\label{eq:NCLT}
\frac{1}{N^H}\sum_{n=1}^{[Nt]}\Big( X(n)-\E X(n)\Big) \Rightarrow Z(t):=\sum_{r=0}^{[k/2]} d_{k,r} Z_{k-2r}(t),
\end{equation}
where
\begin{equation}\label{eq:H formula}
H=\alpha+k/2+1,
\end{equation}
\begin{equation}\label{eq:d_kr}
d_{k,r}=\frac{k!}{2^r(k-2r)!r!},
\end{equation}
$Z_0(t):=0$, and if $k-2r>0$,
\[
Z_{k-2r}(t):=\int'_{\mathbb{R}^{k-2r}} \int_0^t g_r(s\mathbf{1}-\mathbf{x})\mathrm{1}_{\{s\mathbf{1}>\mathbf{x}\}}ds ~B(dx_1)\ldots B(dx_{k-2r})
\]
is a $(k-2r)$-th order generalized Hermite process and
\begin{equation}\label{eq:g_r(s1-x)}
g_r(s\mathbf{1}-\mathbf{x}):=\int_{\mathbb{R}_+^{r}}g(y_1,y_1,\ldots,y_r,y_r,s-x_{1},\ldots,s-x_{k-2r}) dy_1\ldots dy_r.
\end{equation}
\end{Thm}

\begin{proof}
The process $X(n)$ is well-defined in the $L^2(\Omega)$-sense by Proposition \ref{Pro:X(n) well defined}. We now use the notation in Proposition \ref{Pro:basic acf}. If the basic off-diagonal term $X_{\pi}^{\mathbf{j}}(n)$ in (\ref{eq:X_c}) satisfies $m+r<k$,  in view of that proposition and the proof of Corollary \ref{Cor:full acf}, one has 
$$
N^{-H} \Var\Big[\sum_{n=1}^{[Nt]}  X_{\pi}^{\mathbf{j}}(n)\Big]\rightarrow 0,
 $$
 as $N\rightarrow \infty$. So these terms  converge in probability to zero in $D[0,1]$.

Suppose now  that $m+r=k$ or equivalently $k-2r=m-r$. The goal is show the weak convergence in $D[0,1]$ of  $N^{-H}\sum_{n=1}^{[Nt]}X_{\pi}^{\mathbf{j}}(n)$ to $Z_{k-2r}(t)$. The tightness is standard since $H>1/2$ (see, e.g., Proposition 4.4.2 of \citet{giraitis:koul:surgailis:2009:large}). It remains to show convergence of the finite-dimensional distributions. To do so, we will use Proposition \ref{Pro:Poly->Wiener}, which only requires to show that the convergence in (\ref{eq:htilde h L2 conv}) holds separately for  each order $k-2r=m-r$ with $r=1,\ldots,[k/2]$ and for a single $t>0$.

For simplicity, we assume $a(\cdot)=g(\cdot)$ (including a general $L$ satisfying (\ref{eq:L assump}) is easy), and further one can assume without loss of generality that
$$ 
a_\pi(i_1,\ldots,i_m)=g(i_1,i_1,\ldots,i_r,i_r,i_{r+1},\ldots,i_m),
$$
and thus $X_{\pi}^{\mathbf{j}}(n)$ is as given in (\ref{eq:basic contri}). Let 
$$  
(\mathbf{i},\mathbd{\ell})=(i_1,i_1,\ldots,i_r,i_r,l_1,\ldots,l_{m-r}),
 $$
 and since $X_\pi^{\mathbf{j}}(n)$ has no diagonals, we let
\[
F(\mathbd{\ell},n)=\{\mathbf{i}\in \mathbb{R}_+^{r}: i_u\neq i_v \text{ for }u\neq v;~\text{and } i_p\neq n-l_q \text{ including the case }p=q\},
\] so that we can write
\begin{align*}
N^{-H}\sum_{n=1}^{[Nt]} X_{\pi}^{\mathbf{j}}(n)&=\sum_{\mathbd{\ell}\in \mathbb{R}_+^{m-r}}'\frac{1}{N^{\alpha+k/2+1}}\sum_{n=1}^{[Nt]}\sum_{\mathbf{i}\in F(\mathbd{\ell},n)} g\left(\mathbf{i},n\mathbf{1}-\mathbd{\ell}\right)\mathrm{1}_{\{n\mathbf{1}>\mathbd{\ell}\}}\epsilon_{l_1}\ldots\epsilon_{l_{m-r}}=:Q_{m-r}(h_{t,N}).
\end{align*}
By associating $\mathbf{i}$ to $[N\mathbf{x}]+\mathbf{1}$ and $n$ to $[Ns]+1$, we write the inner sums into integrals, namely,
\begin{align}
h_{t,N}(\mathbd{\ell})&=\frac{1}{N^{\alpha+k/2+1}}\sum_{n=1}^{[Nt]}\sum_{\mathbf{i}\in F(\mathbd{\ell},n)} g\left(\mathbf{i},n\mathbf{1}-\mathbd{\ell}\right)\mathrm{1}_{\{n\mathbf{1}>\mathbd{\ell}\}}\label{eq:h_t,N(l)}
\\&=N^{-(m-r)/2}\int_0^tds\int_{\mathbb{R}_+^r}d\mathbf{x} g\left(\frac{[N\mathbf{x}]+\mathbf{1}}{N},\frac{[Ns]\mathbf{1}+\mathbf{1}-\mathbd{\ell}}{N}\right)\mathrm{1}_{\{[Ns]+\mathbf{1}>\mathbd{\ell}\}} 1_{G(\mathbd{\ell},N)}-R_{t,N}(\mathbd{\ell}),\label{eq:h_t,N(l) expand}
\end{align}
where
\[G(\mathbd{\ell},N)=\{[Nx_u]\neq [Nx_v]\text{ for }u\neq v; \text{ and} ~[Nx_p]\neq [Ns]-l_q \text{ including the case }p=q \},\]
and where
\[
R_{t,N}(\mathbd{\ell})=N^{-(m-r)/2} \frac{Nt-[Nt]}{N}\int_{\mathbb{R}_+^r}d\mathbf{x} g\left(\frac{[N\mathbf{x}]+\mathbf{1}}{N},\frac{[Nt]\mathbf{1}+\mathbf{1}-\mathbd{\ell}}{N}\right)\mathrm{1}_{\{[Nt]+\mathbf{1}>\mathbd{\ell}\}}1_{G(\mathbd{\ell},N)}
\]
is a residual term which will be asymptotically negligible\footnote{The presence of $R_{t,N}$ is due to the fact that the sum over $n$ in (\ref{eq:h_t,N(l)}) goes up to $[Nt]$, whereas the corresponding integral in (\ref{eq:h_t,N(l) expand}), before the change of variable, goes up to $Nt$. Such a remainder $R_{t,N}$ which turns out to be asymptotically negligible in $L^2$, should also have been included in \citet{bai:taqqu:2013:generalized}, following Relation (38) and also in the proof of Theorem 6.10 of that paper.}.

In view of Proposition \ref{Pro:Poly->Wiener}, it is sufficient to show that
\begin{equation}\label{eq:goal}
\lim_{N\rightarrow\infty}\|\tilde{h}_{t,N}-h_{t}\|_{L^2(\mathbb{R}^{m-r})}=0,
\end{equation}
where
\[h_t(\mathbf{y})=\int_{0}^t g_r(s\mathbf{1}-\mathbf{y})\mathrm{1}_{\{s\mathbf{1}>\mathbf{y}\}}ds,
\]
and
\begin{equation}\label{eq:tilde h_t,N}
\tilde{h}_{t,N}(\mathbf{y})=N^{(m-r)/2}h_{t,N}([N\mathbf{y}]+\mathbf{1})=\int_0^t ds\int_{\mathbb{R}_+^r}d\mathbf{x} g\left(\frac{[N\mathbf{x}]+\mathbf{1}}{N},\frac{[Ns]\mathbf{1}-[N\mathbf{y}]}{N}\right)\mathrm{1}_{\{[Ns]\mathbf{1}>[N\mathbf{y}]\}}1_{H(N)}- \tilde{R}_{t,N}(\mathbf{y}),
\end{equation}
where
\[
\tilde{R}_{t,N}(\mathbf{y})= \frac{Nt-[Nt]}{N}\int_{\mathbb{R}_+^r}d\mathbf{x} g\left(\frac{[N\mathbf{x}]+\mathbf{1}}{N},\frac{[Nt]\mathbf{1}-[N\mathbf{y}]}{N}\right)\mathrm{1}_{\{[Nt]\mathbf{1}>[N\mathbd{y}]\}}1_{H(N)},
\]
with
\[
H(N)=\{[Nx_u]\neq [Nx_v] \text{ for }u\neq v; \text{ and} ~[Nx_p]\neq [Ns]-[Ny_q]-1 \text{ including the case }p=q\}.\]
The term $H(N)$ comes from $G(\mathbd{\ell},N)$.

We first deal with the term involving $g(\cdot)$ in (\ref{eq:tilde h_t,N}), and then with $\tilde{R}_{t,N}$.
First, the a.e.\ convergence of
\begin{align*}
g\left(\frac{[N\mathbf{x}]+\mathbf{1}}{N},\frac{[Ns]\mathbf{1}-[N\mathbf{y}]}{N}\right)\mathrm{1}_{\{[Ns]\mathbf{1}>[N\mathbf{y}]\}}1_{ H(N)}
\end{align*}
to
\begin{align*}
g(\mathbf{x},s\mathbf{1}-\mathbf{y})\mathrm{1}_{\{s\mathbf{1}>\mathbf{y}\}}~~\text{ as } N\rightarrow\infty
\end{align*}
 follows from the a.e.\ continuity of $g$, and the a.e.\ convergence of $1_{H(N)}$ to $1$
\footnote{If $L$ had not been taken to be $1$, we would have an additional term $L([N\mathbf{x}]+\mathbf{1},[Ns]\mathbf{1}-[N\mathbf{y}])$, which tends to $1$ as $N\rightarrow\infty$ by (\ref{eq:L assump}). Indeed, we can write $
\Big([N\mathbf{x}]+\mathbf{1},[Ns]\mathbf{1}-[N\mathbf{y}]\Big)=\Big([N\mathbf{x}],[N(s\mathbf{1}-\mathbf{y})]\Big)+\mathbf{B}_N(s,\mathbf{y}),$
where for each fixed $(s,\mathbf{y})$, the remainder
$
\mathbf{B}_N(s,\mathbf{y})=\Big(\mathbf{1},[Ns]\mathbf{1}-[N\mathbf{y}]- [N(s\mathbf{1}-\mathbf{y})]\Big)
$
is bounded with respect to $N$, thus we apply (\ref{eq:L assump})}.  We are thus left to establish suitable bounds in order to apply  the Dominated Convergence Theorem.

By the definition of a GHK(B), 
$$
|g(\mathbf{x})|\le c\|\mathbf{x}\|^\alpha=c (x_1+\ldots+x_k)^\alpha
 $$
 for some constant $c>0$. Recall that  $\alpha<-k/2$. We hence claim that for any $\mathbf{x}>\mathbf{0}$,
\begin{equation}\label{eq:dom bound 1}
\left|g\left(\frac{[N\mathbf{x}]+\mathbf{1}}{N},\frac{[Ns]\mathbf{1}-[N\mathbf{y}]}{N}\right)\right|\mathrm{1}_{\{[Ns]\mathbf{1}>[N\mathbf{y}]\}}\le g^*\left(\mathbf{x},s\mathbf{1}-\mathbf{y}\right)\mathrm{1}_{\{s\mathbf{1}>\mathbf{y}\}},
\end{equation}
where 
$$
g^*(\mathbf{z})=C\|\mathbf{z}\|^\alpha=C(|z_1|+\ldots+|z_m|)^\alpha,\quad \mathbf{z}\in \mathbb{R}^m,
$$
 for some constant $C>0$.
This is because $\{[Ns]\mathbf{1}>[N\mathbf{y}]\} \subset \{s\mathbf{1}>\mathbf{y}\}$, and on the set $\{\mathbf{x}>\mathbf{0},[Ns]\mathbf{1}>[N\mathbf{y}]\}$, we have $([N\mathbf{x}]+1)/N>\mathbf{x}$, as well as  $([Ns]-[Ny_j])/N \ge \frac{1}{2}(s-y_j)$ (see Relation (40) in the proof of Theorem 6.5 of \citet{bai:taqqu:2013:generalized}), $j=1,\ldots,m-r$. But by Remark \ref{Rem:int|g(s1-x)|ds<inf}, for any $t>0$ and a.e.\ $\mathbf{y}\in \mathbb{R}^{m-r}$,
\begin{align}\label{eq:dom bound 2}
\int_0^t ds \int_{\mathbb{R}_+^r} d\mathbf{x} g^*(\mathbf{x},s\mathbf{1}-\mathbf{y})\mathrm{1}_{\{s\mathbf{1}>\mathbf{y}\}}=\int_0^t g^*_r(s\mathbf{1}-\mathbf{y})\mathrm{1}_{\{s\mathbf{1}>\mathbf{y}\}} ds<\infty,
\end{align}
where $g_r^*(\mathbf{y})=C'\|\mathbf{y}\|^{\alpha+r}$ for some $C'>0$ is a GHK on $\mathbb{R}_+^{m-r}$ (see Lemma \ref{Lem:int g}). One hence obtains by (\ref{eq:dom bound 1}), (\ref{eq:dom bound 2}) and the Dominated Convergence Theorem that $\tilde{h}_{t,N}(\mathbf{y})$ converges to $h_{t}(\mathbf{y})$ for a.e. $\mathbf{y}\in \mathbb{R}^{m-r}$. To conclude the $L^2$-convergence of $\tilde{h}_{t,N}$ to $h_{t}$, note that
\begin{equation}\label{eq:bound outer}
|\tilde{h}_{t,N}(\mathbf{y})|\le \tilde{h}^*_{t}(\mathbf{y}):=\int_0^t ds  g^*_r\left(s\mathbf{1}-\mathbf{y}\right)\mathrm{1}_{\{s\mathbf{1}>\mathbf{y}\}},
\end{equation}
where $\tilde{h}^*_{t} \in L^2(\mathbb{R}^{m-r})$ by Remark \ref{Rem:int|g(s1-x)|ds<inf}. Since $h_t(\mathbf{y})\in L^2(\mathbb{R}^{m-r})$ as well, we can apply the $L^2$-version Dominated Convergence Theorem to conclude (\ref{eq:goal}),  because  the remainder term $\tilde{R}_{t,N}$ in (\ref{eq:tilde h_t,N}) satisfies
\begin{align*}
\|\tilde{R}_{t,N}(\mathbf{y})\|_{L^2(\mathbb{R}^{m-r})}^2\le& \left(\frac{Nt-[Nt]}{N}\right)^2 N^{-(m-r)} \sum_{\mathbd{l}>\mathbf{0}}\left( \int_{\mathbb{R}^r_+}d\mathbf{x} g^*\left(\mathbf{x}, \frac{\mathbd{\ell}}{N} \right) \right)^2
\\=& N^{-2H} (Nt-[Nt])^2  \sum_{\mathbf{\ell}>\mathbf{0}} g_r^*\left(\mathbd{\ell}\right)^2\rightarrow 0
\end{align*}
as $N\rightarrow\infty$ since $\sum_{\mathbd{\ell}>\mathbf{0}} g_r^*\left(\mathbf{i}\right)^2<\infty$. We also used the fact that $ \int_{\mathbb{R}^r_+}d\mathbf{x} g^*\left(\mathbf{x}, \frac{\mathbd{\ell}}{N}\right)=N^{-\alpha-r}g^*_r(\mathbd{\ell})$ and $H=\alpha+(m+r)/2+1$.

Finally, the combinatorial coefficients $d_{k,r}$ in (\ref{eq:NCLT}) are obtained by counting the ways  of choosing $r$ subsets out of the $k$ variables, where each subset contains $2$ variables, and where the order of the $r$ subsets does not matter. One can apply the multinomial formula involving $k$ variables to be divided into one group of $k-2r$ variables and $r$ groups of $2$ variables, but since the order of these $r$ groups is irrelevant, there is an additional division by $r!$. Hence
\[
d_{k,r}= \frac{k!}{(k-2r)!(2!)^r }\frac{1}{r!}.
\]
\end{proof}

\begin{Rem}
We have considered only causal forms because  for the coefficient $a(\mathbf{i})$ for non-causal forms, $\mathbf{i}\in \mathbb{Z}^k$, one can specify different homogeneity exponents $\alpha$ in different orthotopes of $\mathbf{i}$ for $a(\mathbf{i})$, and only the orthotope with the highest $\alpha$ will contribute in the limit.
\end{Rem}

\begin{Eg}\label{Eg}
Set in Theorem \ref{Thm:Basic NCLT}  $a(\mathbf{i})=g(\mathbf{i})=(i_1+i_2+i_3+i_4+i_5)^{\alpha}$, where $-3<\alpha<-5/2$. Hence by (\ref{eq:H formula}), 
$$  
H=\alpha+5/2+1=\alpha+7/2\in (1/2,1).
$$
 That is, we consider
\[
X(n)=\sum_{0<i_1,\ldots,i_5<\infty} (i_1+i_2+i_3+i_4+i_5)^{\alpha} \epsilon_{n-i_1}\epsilon_{n-i_2} \epsilon_{n-i_3} \epsilon_{n-i_4}  \epsilon_{n-i_5}.
\]
Here $k=5$, and hence $r$, which denotes the possible number of pairings of variables, can be $0$, $1$, or $2$. The corresponding functions $g_r$'s in  (\ref{eq:g_r}), are $g_0=g$, where no pairing takes place,
\[
g_1(x_1,x_2,x_3)=\int_0^\infty(x_1+x_2+x_3+2y)^{\alpha}dy= \frac{1}{2(\alpha+1)}(x_1+x_2+x_3)^{\alpha+1},
\]
where there is one pairing, and
\[
g_2(x_1)=\int_0^\infty\int_0^\infty (x_1+2y_1+2y_2)^{\alpha}dy_1dy_2=\frac{1}{4(\alpha+1)(\alpha+2)}x_1^{\alpha+2},
\]
where there are two pairings.
Moreover, by (\ref{eq:d_kr}), 
$$  
d_{5,0}=1,\quad  d_{5,1}=\frac{5!}{2\times 3!\times 1}=10,\quad  \mbox{\rm and} \quad d_{5,2}=\frac{5!}{2^2\times 1\times 2}=15.
 $$
 We have then the following convergence in $D[0,1]$:
\begin{align*}
\frac{1}{N^H}\sum_{n=1}^{[Nt]}[X(n)-\E X(n)]\Rightarrow Z_{5}(t)+10Z_3(t)+15Z_1(t),
\end{align*}
where
\[
Z_5(t):=\int'_{\mathbb{R}^5} \int_0^t (5s- x_1-x_2-x_3-x_4-x_5)^{\alpha} \mathrm{1}_{\{s\mathbf{1}>\mathbf{x}\}}ds ~B(dx_1)B(dx_2)B(dx_3)B(dx_4)B(dx_5),
\]
\[
Z_3(t):=\frac{1}{2(\alpha+1)}\int'_{\mathbb{R}^3} \int_0^t (3s- x_1-x_2-x_3)^{\alpha+1} \mathrm{1}_{\{s\mathbf{1}>\mathbf{x}\}}ds ~B(dx_1)B(dx_2)B(dx_3)
\]
and
\[
Z_1(t):=\frac{1}{4(\alpha+1)(\alpha+2)}\int_{\mathbb{R}} \int_0^t (s- x_1)_+^{\alpha+2} ds ~B(dx_1).
\]
Observe that $Z_1(t)$ is fractional Brownian motion with $H=\alpha+7/2$, and can be expressed as
\[
Z_1(t)=\frac{1}{4(\alpha+1)(\alpha+2)(\alpha+3)} \int_{\mathbb{R}}\left[ (t-x_1)_+^{\alpha+3}-(-x_1)_+^{\alpha+3}\right] B(dx_1).
\]
\end{Eg}

\section{Expressing the NCLT limit as a centered multiple\\ Wiener-Stratonovich integral}\label{Sec:Stra}

When Norbert Wiener (see, e.g, ,\citet{wiener:1966:nonlinear}) first introduced the multiple integral with respect to a Brownian motion, he did not  exclude the diagonals to render integrals of different orders orthogonal to each other, although the  idea of orthogonalization was in fact informally developed (see Lecture 4 in \citet{wiener:1966:nonlinear}). \citet{ito:1951:multiple} modified Wiener's definition by excluding the diagonals, and made the $k$-tuple integral $I_k(f)$ well-defined for all $f\in L^2(\mathbb{R}^k)$. Since then, the literature had focused on It\^o's off-diagonal integrals. \citet{hu:meyer:1988:integrales}, however, considered integrals with diagonals and related them to the iterated Stratonovich integrals.  Formal theories were later developed in \citet{johnson:kallianpur:1993:homogeneous} and \citet{budhiraja:1994:hilbert}.

We denote the $k$-tuple Wiener-Stratonovich integral as  $\mathring{I}_k(\cdot)$. The integral $\mathring{I}_k(\cdot)$ and the Wiener-It\^o integral $I_k(\cdot)$ are related through the following Hu-Meyer formula: for a symmetric function $h\in L^2(\mathbb{R}^k)$,
\begin{equation}\label{eq:Hu Meyer}
\mathring{I}_k(h)= \sum_{r=0}^{[k/2]}{d_{k,r}} I_{k-2r}(\tau^r h),
\end{equation}
where $d_{k,r}$ is as in (\ref{eq:d_kr}), and $\tau^r$ is the so-called $r$-th $\tau$-trace defined as
\[
(\tau^r h)(\mathbf{x})=\int_{\mathbb{R}^r} h(y_1,y_1,y_2,y_2,\ldots,y_r,y_r, x_1,\ldots ,x_{k-2r}) d\mathbf{y},
\]
provided that  $\tau^r(h)\in L^2(R^{k-r})$  (see Definition 2.7 of \citet{budhiraja:1994:hilbert}). In the integral defining $\tau^r(h)$, we have $r$ pairs of $y$'s. We note that the formula (\ref{eq:Hu Meyer}) was in fact known to Wiener (see (5.14) of \citet{wiener:1966:nonlinear}).  There is also a more general notion of trace than $\tau^r$, called the limiting trace and denoted by $\overrightarrow{\mathrm{Tr}}^r$ (see Definition 2.3 of \citet{budhiraja:1994:hilbert}), involving tensor products of Hilbert space. It is sufficient  for our purpose to focus  on the $\tau$-trace. Note that  if $k$ is even and $r=k/2$, then $r-k/2=0$ in (\ref{eq:Hu Meyer}). In addition, the following convention is used:
$$I_0(\tau^{k/2}(h)):=\tau^{k/2}(h)=\E \mathring{I}(h).
$$

A heuristic understanding of the Hu-Meyer formula (\ref{eq:Hu Meyer}) is as follows. In  the integral
$$\mathring{I}_k(\cdot)=\int_{\mathbb{R}^k} ~\cdot~ B(dx_1)\ldots B(dx_k)$$
which includes the diagonals (we do not have the prime $'$ on the integral symbol), let's restrict first the integration set to $$
\{x_1=x_2=x,~ x\neq x_p, x_p\neq x_q,~ p\neq q\in \{3,\ldots,k\}\}.$$
The  integrator $B(dx_1)B(dx_2)$ then becomes $B(dx)^2=[B(dx)^2-dx] +dx$. The first term $[B(dx)^2-dx]$, whose variance is $2(dx)^2$, yields the integral $\int_{\mathbb{R}^{k-1}}' h(x,x,x_3,\ldots,x_k) [B(dx)^2-dx]B(dx_3)\ldots B(dx_{k})$ with variance $\int_{\mathbb{R}^{k-2}}\int_{\mathbb{R}} h(x,x,x_3,\ldots,x_k)^2 [2(dx)^2] dx_3\ldots dx_k =0$, because we have a higher power of $dx$ than needed. This means that if we integrate on the set indicated above, we end up with
$$\int_{\mathbb{R}^{k-2}}\int_{\mathbb{R}} h(x,x,x_3,\ldots,x_k) dx~ B(dx_3)\ldots B(dx_{k}).$$
If moreover, we integrate on the set $x_1=\ldots=x_l=x$, $l\ge 3$ and all $x$, $x_{l+1},\ldots,x_{k}$  distinct, using the fact $\E( B(dx)^{2n})= (2n-1)!! (dx)^{n}$, it can be shown that one always ends up with higher power of $dx$ than needed, and these terms are thus all zero. Hence the only way of getting terms that really contribute is to identify only pairs of the variables, which results in the form (\ref{eq:Hu Meyer}).

To express the limits in  Theorem \ref{Thm:Basic NCLT} in terms of Wiener-Stratonovich integrals,  let
\begin{equation}\label{eq:h_tx}
h_t(\mathbf{x})=\int_0^t g(s\mathbf{1}-\mathbf{x})\mathrm{1}_{\{s\mathbf{1}>\mathbf{x}\}}ds,
\end{equation}
where $g$ is a GHK(B) on $\mathbb{R}_+^k$. Suppose that   $2r<k$, which is always the case when $k$ is odd. Then
\begin{align*}
\tau^r h_t(\mathbf{x}) &=\int_{\mathbb{R}^r}d\mathbf{y}\int_0^t ds g(s-y_1,s-y_1,\ldots,s-y_r,s-y_r,s-x_1,\ldots,s-x_{k-2r})1_{\{s\mathbf{1}>\mathbf{x}\}}1_{\{s\mathbf{1}>\mathbf{y}\}}\\
&=\int_0^tds\int_{\mathbb{R}_+^{r}}d\mathbf{y}g(y_1,y_1,\ldots,y_r,y_r,s-x_{1},\ldots,s-x_{k-2r})1_{\{s\mathbf{1}>\mathbf{x}\}} 
\\
&= \int_0^t g_r(s\mathbf{1}-\mathbf{x})\mathrm{1}_{\{s\mathbf{1}>\mathbf{x}\}} ds,
\end{align*}
  where $g_r(s\mathbf{1}-\mathbf{x})$ is as given in (\ref{eq:g_r(s1-x)}), and the change of integration order is  justified by Fubini as  the proof of Lemma \ref{Lem:int g}. (Observe that $2r<k$ is an assumption of Lemma \ref{Lem:int g}.)

In the special case when $k$ is even and $2r=k$, the change of the integration order cannot be justified by Fubini, and
$$
\tau^{k/2}(h_t)=\int_{\mathbb{R}^{k/2}}d\mathbf{y}\int_0^t g(s-y_1,s-y_1,\ldots,s-y_{k/2},s-y_{k/2})1_{\{s\mathbf{1}>\mathbf{y}\}}ds
$$
may not exist, because for example if $g(\mathbf{x})=\|\mathbf{x}\|^{\alpha}$, then
$$
\int_{\mathbb{R}^{k/2}_+}\|(x_1,x_1,\ldots,x_{k/2},x_{k/2})\|^{\alpha}d\mathbf{x}=c\int_{\mathbb{R}^{k/2}_+} (x_1+\ldots+x_{k/2})^\alpha d\mathbf{x}=\infty.
$$
 Theorem \ref{Thm:Basic NCLT}, however, does not involve the term $r=k/2$ (the zero-order term $Z_0(t)$ was defined to be zero). So we cannot directly use the Hu-Meyer formula (\ref{eq:Hu Meyer}) to re-express the   limit process  $Z(t)$ in (\ref{eq:NCLT}) in terms of a multiple Wiener-Stratonovich integral. We have to consider instead the \emph{centered Wiener-Stratonovich integral} which we define as     
\begin{equation}\label{eq:centered Hu Meyer}
\mathring{I}^c_k(h)=\sum_{0\le r< k/2}{d_{k,r}} I_{k-2r}(\tau^r h),
\end{equation}
where we do not include the  $0$-th order (constant) term which arises when $r=k/2$. Consequently, the integral has always  $0$ mean. Note that $\mathring{I}^c_k$ coincides $\mathring{I}_k$ when $k$ is odd, but obviously admits a larger class of integrands $h$ when $k$ is even.
With this modification, we are able to restate Theorem \ref{Thm:Basic NCLT} for the  long-memory Volterra process as follows.
\begin{Cor}
Let $X(n)=\sum_{\mathbf{i}\in \mathbb{Z}_+^k} g(\mathbf{i})L(\mathbf{i})\epsilon_{n-i_1}\ldots\epsilon_{n-i_k}$ be as in Theorem \ref{Thm:Basic NCLT}. Then
\begin{equation*}
\frac{1}{N^H}\sum_{n=1}^{[Nt]}\left[X(n)-\E X(n)\right] \Rightarrow Z(t)=\mathring{I}^c_k(h_t)=\sum_{0\le r< k/2}{d_{k,r}} I_{k-2r}(\tau^r h_t),
\end{equation*}
where $h_t$ is defined in (\ref{eq:h_tx}).
\end{Cor}

\section{An extended hypercontractivity formula}\label{Sec:Extension hyper}
Suppose that $\mathbd{\epsilon}_i=(\epsilon_{i}^{(1)},\ldots,\epsilon_{i}^{(k)})$ are i.i.d.\ vectors with $0$ means and finite variances. Let $a(\cdot)$ be a function defined on the tetrahedral $\{(i_1,\ldots,i_k)\in \mathbb{Z}_+^k,~ i_1<\ldots<i_k\}$, and let
\[
T_n^k (a)=\sum_{1\le i_1<\ldots<i_k\le n} a(i_1,\ldots,i_k)\epsilon_{i_1}^{(1)}\ldots\epsilon_{i_k}^{(k)},
\]
where the case $k=0$ is understood as a constant $a$.  For $k\ge 1$, define
\[
T_{i-1}^{k-1}(a)=\sum_{1\le i_1<\ldots<i_{k-1}\le i-1} a(i_1,\ldots,i_{k-1},i)\epsilon_{i_1}^{(1)}\ldots\epsilon_{i_{k-1}}^{(k-1)}.
\]
Then
\begin{equation}\label{eq:mart diff}
T_n^k(a)=\sum_{i=k}^n T_{i-1}^{k-1}(a)\epsilon_{i}^{(k)}.
\end{equation}
Let $\mathcal{F}_n=\sigma(\mathbd{\epsilon}_i, i\le n)$. Then $\{T_n^k(a), n\ge k\}$ is a martingale with respect to $\mathcal{F}_n$ and (\ref{eq:mart diff}) is a decomposition into martingale differences since 
$$ 
\E[ T_{n-1}^{k-1}(a)\epsilon_{n}^{(k)}|\mathcal{F}_{n-1}]
=T_{n-1}^{k-1}(a)\E[\epsilon_{n}^{(k)}]=0.
$$
\begin{Lem}[\citet{burkholder:1973:distribution}]
Let $p>2$,  and let $X_i$ be martingale differences. Then for some constant $C_p>0$, we have
\[
\left\|\sum_{i=1}^n X_i\right\|_p \le C_p \left\|\left(\sum_{i=1}^n X_i^2\right)^{1/2}  \right\|_p.
\]
\end{Lem}

\begin{Lem}
Let $p>2$,
and let $Y_i$'s be random variables such that $\E |Y_i|^p<\infty$. Then
\[
\left\|\left(\sum_{i=1}^n Y_i^2\right)^{1/2}\right\|_p\le
\left(\sum_{i=1}^n \|Y_i\|_p^2\right)^{1/2}.
\]
\end{Lem}
\begin{proof}
By   Minkowski's inequality,
\begin{align*}
\left\|\left(\sum_{i=1}^n Y_i^2\right)^{1/2}\right\|_p^2=
\left[\E\left(\sum_i Y_i^2\right)^{p/2}\right]^{2/p}=\left\|\sum_{i=1}^n Y_i^2\right\|_{p/2}
\le  \sum_{i=1}^n \| Y_i^2 \|_{p/2}= \sum_{i=1}^n (\E |Y_i|^p)^{2/p}=\sum_{i=1}^n \| Y_i\|_p^2.
\end{align*}

\end{proof}

The following result is used in Lemma \ref{Lem:hypercontract}.
\begin{Thm} If $p>2$,
and $\E |\epsilon_{i}^{(j)}|^p<\infty$, then
\[
\|T_n^k(a)\|_p \le c\|T_n^k(a)\|_2
\]
where $c$ is a constant that does not depend on $a(\cdot)$ nor on $n$.
\end{Thm}
\begin{proof}
We prove it by induction. The case $k=0$ is trivial since $T_n^0(a)$  is a constant.  Suppose that the inequality holds for $k-1$, where $k\ge 1$. Then by the forgoing lemmas,
\begin{align*}
\|T_n^k(a)\|_p 
&\le C_p \left\|\left(\sum_{i=k}^n |T_{i-1}^{k-1}(a) \epsilon_i^{(k)}|^2\right)^{1/2}\right\|_p\\
&\le C_p \left(\sum_{i=k}^n \|T_{i-1}^{k-1}(a)\epsilon_i^{(k)}\|_p^2\right)^{1/2}\\
&= C_p \left(\sum_{i=k}^n \|T_{i-1}^{k-1}(a)\|_p^2 \|\epsilon_i^{(k)}\|_p^2\right)^{1/2},
\end{align*}
by independence between $T_{i-1}^{k-1}(a)$ and $\epsilon_i^{(k)}$.
By the induction assumption, $\|T_{i-1}^{k-1}(a)\|_p^2\le c_1\|T_{i-1}^{k-1}(a)\|_2^2 $ for some $c_1>0$ which does not depend on $a(\cdot)$ or $n$. In addition,  trivially since the random vectors $\{\mathbd{\epsilon}_i, i\in \mathbb{Z}\}$ are identically distributed, one has $\|\epsilon_i^{(k)}\|_p^2\le c_2 \|\epsilon_i^{(k)}\|_2^2 $ for some $c_2>0$ which does not depend on $a(\cdot)$, $n$ or $i$. The desired result is  then immediate once noting that
\[
\|T_n^k(a)\|_2^2=\sum_{i=k}^n \|T_{i-1}^{k-1}(a)\|_2^2 \|\epsilon_{i}^{(k)}\|_2^2
\]
due to the off-diagonality of $T_n^k(a)$ and independence.
\end{proof}

\noindent\textbf{Acknowledgments.} We would like to thank the referees for their comments. This work was partially supported by the NSF grant DMS-1007616 and DMS-1309009 at Boston University.

\bibliographystyle{plainnat}

\medskip
\noindent Shuyang Bai~~~~~~~ \textit{bsy9142@bu.edu}\\
Murad S. Taqqu ~~\textit{murad@bu.edu}\\
Department of Mathematics and Statistics\\
111 Cumminton Street\\
Boston, MA, 02215, US
\end{document}